\documentclass[11pt,reqno]{amsart}

\usepackage{ifthen,epic,eepic,color,stmaryrd,bm}
\usepackage{mathrsfs}
\usepackage{amssymb}
\usepackage{amsthm}
\usepackage{amsmath,amsfonts,amssymb,esint}
\usepackage{graphics,color}
\usepackage{enumerate}
\usepackage{marginnote}
\usepackage{xfrac}

\newtheorem{theorem}{Theorem}[section]

\newtheorem{proposition}[theorem]{Proposition}

\newtheorem{definition}[theorem]{Definition\rm}
\newtheorem{remark}[theorem]{Remark}

\newcommand{\T}{\ensuremath{\mathbb{T}}}
\newcommand*{\R}{\ensuremath{\mathbb{R}}}
\renewcommand*{\S}{\ensuremath{\mathcal{S}}}
\newcommand*{\N}{\ensuremath{\mathbb{N}}}
\newcommand*{\Z}{\ensuremath{\mathbb{Z}}}

\renewcommand*{\div}{\ensuremath{\mathrm{div\,}}}
\newcommand*{\tr}{\ensuremath{\mathrm{tr\,}}}

\newcommand{\eps}{\varepsilon}

\newcommand*{\hys}{|||}

\renewcommand*{\P}{\ensuremath{\mathcal{P}}}
\newcommand*{\Q}{\ensuremath{\mathcal{Q}}}
\newcommand*{\RR}{\ensuremath{\mathcal{R}}}

\begin{document}

\title[Onsager's conjecture]
{Dissipative Euler flows and Onsager's conjecture}

\author{Camillo De Lellis}
\address{Institut f\"ur Mathematik, Universit\"at Z\"urich, CH-8057 Z\"urich}
\email{camillo.delellis@math.unizh.ch}

\author{L\'aszl\'o Sz\'ekelyhidi Jr.}
\address{Institut f\"ur Mathematik, Universit\"at Leipzig, D-04103 Leipzig}
\email{laszlo.szekelyhidi@math.uni-leipzig.de}

\begin{abstract}
Building upon the techniques introduced in \cite{DS3}, for any $\theta<\frac{1}{10}$ we construct periodic weak solutions of the incompressible Euler equations
which dissipate the total kinetic energy and are H\"older-continuous with exponent $\theta$.
A famous conjecture of Onsager states the existence of such dissipative solutions with any H\"older exponent $\theta<\frac{1}{3}$. Our theorem is the first result in this direction.
\end{abstract}

\maketitle

\section{Introduction}
The Euler equations for the motion of an inviscid perfect fluid are
\begin{equation}\label{e:Euler1}
\left\{\begin{aligned}
&\partial_t v + v\cdot\nabla v + \nabla p =0\\
&\div v = 0
\end{aligned}\right.\; ,
\end{equation}
where $v(x,t)$ is the velocity vector and $p(x,t)$ is the internal pressure. In this paper we 
consider the equations in $3$ dimensions and assume the domain to be periodic, i.e. the $3$-dimensional torus $\T^3 = {\mathbb S}^1\times {\mathbb S}^1 \times {\mathbb S}^1$. 
Multiplying \eqref{e:Euler1} by $v$ itself and integrating, we obtain the formal energy balance
$$
\frac{1}{2}\frac{d}{dt}\int_{\T^3}|v(x,t)|^2\,dx=-\int_{\T^3}[((v\cdot \nabla) v)\cdot v] (x,t)\,dx.
$$
If $v$ is continuously differentiable in $x$, we can integrate the right hand side by parts and conclude that
\begin{equation}\label{e:energycons}
\int_{\T^3}|v(x,t)|^2\,dx=\int_{\T^3}|v(x,0)|^2\,dx\quad\textrm{ for all }t>0.
\end{equation}
On the other hand, in the context of $3$-dimensional turbulence it is important to consider {\em weak solutions}, where $v$ and $p$ are not necessarily differentiable. If $(v,p)$ is merely continuous, one can define weak solutions (see e.g.~\cite{OseenBook,LichtensteinBook}) by integrating \eqref{e:Euler1} over simply connected subdomains $U\subset\T^3$ with $C^1$ boundary, to obtain
the identities
\begin{equation}\label{e:Euler2}
\begin{split}
\int_Uv(x,0)\,dx=\int_Uv(x,t)\,dx+&\int_0^t\int_{\partial U}[v(v\cdot\nu)+p\nu] (x,s)\, dA (x)\, ds\,\\ 
\int_{\partial U}[v\cdot\nu] (x,t)\, dA (x)=&0 
\end{split}
\end{equation}
for all $t$. In these identities $\nu$ denotes the unit outward normal to $U$ on $\partial U$ and $dA$ denotes the usual area element. Indeed, the formulation \eqref{e:Euler2} appears {\em first} in the derivation of the Euler equations from Newton's laws in continuum mechanics, and \eqref{e:Euler1} is then {\em deduced} from \eqref{e:Euler2} for sufficiently regular $(v,p)$. 
It is also easy to see that pairs of continuous functions $(v,p)$ satisfy \eqref{e:Euler2} for all fluid elements $U$ and all
times $t$ if and only if they solve \eqref{e:Euler1} in the "modern" distributional sense
(rewriting the first line as
$\partial_t u + {\rm div}\, (v\otimes v) + \nabla p =0$).

\medskip

For weak solutions, the energy conservation \eqref{e:energycons} might be violated, and indeed, this possibility has been considered for a rather long time in the context of $3$ dimensional turbulence. In his famous note \cite{Onsager} about statistical hydrodynamics, Onsager considered 
weak solutions satisfying the H\"older condition
\begin{equation}\label{e:hoelderest}
|v(x,t)-v(x',t)|\leq C|x-x'|^\theta,
\end{equation}
where the constant $C$ is independent of $x,x'\in\T^3$ and $t$. He conjectured that 
\begin{enumerate}
\item[(a)] Any weak solution $v$ satisfying \eqref{e:hoelderest} with $\theta>\frac{1}{3}$ conserves the energy;
\item[(b)] For any $\theta<\frac{1}{3}$ there exist weak solutions $v$ satisfying \eqref{e:hoelderest} which do not conserve the energy.
\end{enumerate}
This conjecture is also very closely related to Kolmogorov's famous K41 theory \cite{Kolmogorov} for homogeneous isotropic turbulence in $3$ dimensions. We refer the interested reader to \cite{FrischBook,Robert,EyinkSreenivasan}, see also Section \ref{s:spectrum} below. 

Part (a) of the conjecture is by now fully resolved: it has first been considered by Eyink in \cite{Eyink} following Onsager's original calculations and proved by Constantin, E and Titi in \cite{ConstantinETiti}. Slightly weaker assumptions on $v$ (in Besov spaces) were subsequently shown to be sufficient for energy conservation in \cite{RobertDuchon,CCFS2007}. In contrast, until now part (b) of the conjecture remained widely open. In this paper we address specifically this question by proving the following theorem:

\begin{theorem}\label{t:main} Let $e:\mathbb [0,1]\to\R$ be a smooth positive function. For every $\theta<\frac{1}{10}$ 
there is a pair $(v,p)\in C(\T^3\times [0,1])$ with the following properties:
\begin{itemize}
\item $(v,p)$ solves the incompressible Euler equations in the sense \eqref{e:Euler2};
\item $v$ satisfies \eqref{e:hoelderest};
\item the energy satisfies
\begin{equation}\label{e:energy_id}
e(t) = \int_{\T^3} |v(x,t)|^2\, dx\qquad\forall t\in [0,1]\, .
\end{equation}
\end{itemize}
\end{theorem}

This is the first result in the direction of part (b) of Onsager's conjecture, where H\"older-continuous
solutions are constructed. Prior to this result, there have been several constructions of weak solutions violating \eqref{e:energycons} in \cite{Scheffer93,Shnirelman1,Shnirelmandecrease,DS1,DS2}, but the solutions constructed in these papers are not continuous. The ones of \cite{Scheffer93, Shnirelman1} are just square summable functions of time and space, whereas 
the example of \cite{Shnirelmandecrease} was the first to be in the energy space and the constructions of \cite{DS1,DS2} gave bounded solutions. Recently, in \cite{DS3} we have constructed {\it continuous} weak solutions, but no H\"older exponent was given.

\begin{remark}\label{r:time_and_pressure}
In fact our proof of Theorem \ref{t:main} yields some further regularity properties of the pair $(v,p)$.
First of all, our solutions $v$ are H\"older-continuous in {\em space and time}, i.e. there is a constant $C$ such that
\begin{align*}
|v(x,t) - v(x', t')|&\leq C \left(|x-x'|^\theta + |t-t'|^\theta\right)
\end{align*}
for all pairs $(x,t), (x',t')\in \T^3\times [0,1]$.

From the equation $\Delta p = - {\rm div}\, {\rm div}\, (v\otimes v)$ (normalizing
the pressure so that $\int p (x,t)\, dx = 0$) and standard Schauder estimates one can easily derive H\"older regularity
in space for $p$ as well, with H\"older exponent $\theta$. A more careful 
estimate\footnote{personal communication with L.~Silvestre} improves the exponent to
$2\theta$. It is interesting to observe that in fact our
scheme produces pressures $p$ which have that very H\"older regularity in {\em time and space}, namely
\[
|p(x,t) - p(x', t')|\leq C \left(|x-x'|^{2\theta} + |t-t'|^{2\theta}\right)\, .
\]
\end{remark}

\medskip

\subsection{The energy spectrum}\label{s:spectrum}

The energy spectrum $E(\lambda)$ gives the decomposition of the total energy by wavenumber, i.e. 
\[
\int|v|^2\,dx=\int_0^{\infty}E(\lambda)d\lambda.
\]
One of the cornerstones of the K41 theory is the famous Kolmogorov spectrum
\[
E(\lambda)\sim \epsilon^{\sfrac23}\lambda^{-\sfrac53}
\]
for wave numbers $\lambda$ in the inertial range for fully developed 3-dimensional turbulent flows, where $\epsilon$ is the energy dissipation rate. 
For dissipative weak solutions of the Euler equations as conjectured by Onsager, this would be the expected energy spectrum for all $\lambda\in (\lambda_0,\infty)$.

Our construction, based on the scheme and the techniques introduced in \cite{DS3}, allows for a rather precise analysis of the energy spectrum. In a nutshell the scheme can be described as follows. We construct a sequence of (smooth) approximate solutions to the Euler equations $v_k$, where the error is measured by the (traceless part of the) Reynolds stress tensor $\mathring{R}_k$, cf.~\eqref{e:euler_reynolds} and \eqref{e:def_R}. The construction is explicitly given by a formula of the form
\begin{equation}\label{e:scheme}
v_{k+1}(x,t)=v_k(x,t)+W\bigl(v_k(x,t),R_k(x,t);\lambda_k x,\lambda_k t\bigr)+\textrm{corrector}.
\end{equation}
The corrector is to ensure that $v_{k+1}$ remains divergence-free. The vector field $W$ consists of periodic Beltrami flows in the
fast variables (at frequency $\lambda_k$), which are modulated in amplitude and phase depending on $v_k$ and $R_k$.  
More specifically, the amplitude is determined by the error $R_k$ from the previous step, so that
\begin{align}
\|v_{k+1}-v_k\|_0&\lesssim \delta_k^{\sfrac12},\label{e:C0}\\
\|v_{k+1}-v_k\|_1&\lesssim \delta_k^{\sfrac12}\lambda_k,\label{e:C1_intro}
\end{align}
where $\delta_k=\|\mathring{R}_k\|_{C^0}$.

The frequencies $\lambda_k$ are therefore the active modes in the Fourier spectrum of the velocity field in the limit. Since the sequence $\lambda_k$ diverges rather fast, it is natural to think of \eqref{e:scheme} as iteratively defining the Littlewood-Paley pieces at frequency $\lambda_k$. 
Following \cite{Constantin1997} we can then estimate the (Littlewood-Paley-) {\it energy spectrum} in the limit as 
\[
E(\lambda_k)\sim\frac{\langle |v_{k+1}-v_k|^2\rangle}{\lambda_k}
\]
for the active modes $\lambda_k$, where $\langle\cdot\rangle$ denotes the average over the space-time domain. Since $W$ is the superposition of finitely many 
Beltrami modes, we can estimate $\langle|v_{k+1}-v_k|^2\rangle\sim\delta_k$. Thus, both the regularity of the limit and its energy spectrum are determined by the 
rates of convergence $\delta_k\to 0$ and $\lambda_k\to\infty$.

In \cite{DS3} it was shown (cp.~Proposition 2.2 and its proof) that $W$ can be chosen so that
\begin{equation}
\|\mathring{R}_{k+1}\|_{C^0}\leq C(v_k,\mathring{R}_k) \lambda_k^{-\gamma}\label{e:error}
\end{equation}
for some fixed $0<\gamma\leq 1$. By choosing the frequencies $\lambda_k\to\infty$ sufficiently fast, 
$C^0$ convergence of this scheme follows easily. However, in order to obtain a rate on the divergence of $\lambda_k$ we  
need to obtain an estimate on the error in \eqref{e:error} with an explicit dependence on $v_k$ 
and $\mathring{R}_k$. This is achieved in Proposition \ref{p:Reynolds} and forms a key part of the paper. 
Roughly speaking, our estimate has the form
\begin{equation}\label{e:errornew}
\|\mathring{R}_{k+1}\|_{C^0}\lesssim \frac{\delta_k^{\sfrac12}\|v_k\|_{C^1}}{\lambda_k^\gamma},
\end{equation}
with $\gamma\sim \tfrac12$. A first attempt (based on experience with the isometric embedding problem, see below) at obtaining a rate on $\lambda_k$ would then go as follows: in order to decrease the error
in \eqref{e:errornew} by a fixed factor $K>1$ (i.e. $\delta_{k+1}\leq \frac{1}{K}\delta_k$), we choose $\lambda_k$ accordingly, so that
\begin{equation}\label{e:iteratelambda}
\lambda_k^\gamma\sim K \|v_k\|_{C^1}\delta_k^{-\sfrac12}.
\end{equation}
Using \eqref{e:C1_intro} we can then obtain an estimate on $\|v_{k+1}\|_{C^1}$ and iterate. However, it is easy to see that this leads to super-exponential 
growth of $\lambda_k$ whenever $\gamma<1$. From this one can only deduce the energy spectrum $E(\lambda)\sim\lambda^{-1}$ and no H\"older regularity. 

Our solution to this problem is to force a double-exponential convergence of the scheme, see Section \ref{s:iteration}. In this way the finite H\"older regularity in Theorem \ref{t:main} 
as well as the energy spectrum
\begin{equation}\label{e:spectrum}
E(\lambda_k)\lesssim \lambda_k^{-(\sfrac65-\eps)}
\end{equation}
can be achieved, see Remark \ref{r:Physica_Z}. It is quite remarkable, and much akin to the Nash-Moser iteration, that the more rapid (super-exponential) convergence of the scheme 
leads to a better regularity in the limit.

\medskip

 An underlying physical intuition in the turbulence theory is that the flux in the energy cascade should be controlled by local interactions, see \cite{Kolmogorov,Onsager,Eyink,CCFS2007}. A consequence for part (b) of Onsager's conjecture is that in a dissipative solution the active modes, among which the energy transfer takes place, should be (at most) exponentially distributed. Indeed, Onsager explicitly states in \cite{Onsager}  (cp. also  \cite{EyinkSreenivasan}) that this should be the case. 
 
 For the scheme \eqref{e:scheme} in this paper the interpretation is that $\lambda_k$ should increase at most exponentially. As seen in the discussion above, this would only be possible with $\gamma=1$ in the estimate \eqref{e:errornew}. On the other hand, it is also easy to see that with $\gamma=1$ the estimate indeed leads to Onsager's critical $\tfrac13$ H\"older exponent as well as the Kolmogorov spectrum. Indeed, from \eqref{e:iteratelambda} together with \eqref{e:errornew} and \eqref{e:C1_intro} we would obtain $\delta_k\sim K^{-k}$ and 
$\lambda_k\sim K^{\sfrac32k}$, leading to $E(\lambda_k)\sim \lambda_k^{-\sfrac53}$. Thus, our scheme provides yet another route towards understanding the necessity of local interactions as well as towards the Kolmogorov spectrum, albeit one that {\it does not involve considerations on the energy cascade} but is rather based on the {\it ansatz} \eqref{e:scheme}.

\medskip

Onsager's conjecture has also been considered on shell-models \cite{KatzPavlovic,CFP2007,CFP2010}, whose derivation is motivated by the intuition on locality of interactions. Roughly speaking, the Euler equations is considered in the Littlewood-Paley decomposition, but only nearest neighbor interactions in frequency space are retained in the nonlinear term, leading to an infinite system of coupled ODEs. The analogue of both part (a) and (b) of Onsager's conjecture has been proven in \cite{CFP2007,CFP2010}, in the sense that the ODE system admits a unique fixed point which exhibits a decay of (Fourier) modes consistent with the Kolmogorov spectrum. 

Although our Theorem \ref{t:main} and the corresponding spectrum \eqref{e:spectrum} falls short of the full conjecture, it highlights an important feature of the Euler equations that cannot be seen on such shell models: the critical $\tfrac13$ exponent of Onsager is not just the borderline between energy conservation and dissipation in the sense of part (a) and (b) above. For exponents $\theta<\tfrac13$ one should expect an entirely different behavior of weak solutions altogether, namely the type of non-uniqueness and flexibility that usually comes with the $h$-principle of Gromov \cite{Gromov}. 

\medskip

\subsection{$h$-principle and convex integration}
Our iterative scheme is ultimately based on the convex integration technique introduced by Nash in \cite{Nash}
to produce $C^1$ isometric embeddings of
Riemannian manifolds in low codimension, and vastly generalized by Gromov \cite{Gromov}, although several modifications of this technique are required (see the Introduction of \cite{DS3}). Nevertheless, in line with other results proved using a convex integration technique, our construction again adheres to the usual features of the $h$-principle. In particular,
as in \cite{DS3} we are concerned in this paper with the {\it local} aspects of the $h$-principle. For the Euler equations this means that we only treat the case of a periodic space-time domain instead of an initial/boundary value problem. Also, it should be emphasized that although in Theorem \ref{t:main} the existence of {\it one} solution is stated, the method of construction leads to an {\it infinite number} of solutions, as indeed any instance of the $h$-principle does. We refer the reader to the survey \cite{DSSurvey} for the type of (global) results that could be expected even in the current H\"older-continuous setting.

It is of certain interest to notice that in the isometric embedding problem a phenomenon entirely analogous to the
Onsager's conjecture occurs. Namely, if we consider $C^{1, \alpha}$ isometric embeddings in codimension $1$, then
it is possible to prove the $h$-principle for sufficiently small exponents $\alpha$, whereas one can show the absence
of the $h$-principle (and in fact even some rigidity statements) if the H\"older exponent is sufficiently large. This
phenomenon was first observed by Borisov (see \cite{BorisovRigidity1} and \cite{Borisov65}) and proved in greater
generality and with different techniques in \cite{CDSz}. In particular the proofs given in \cite{CDSz} of both the $h$-principle 
and the rigidity statements share many similarities with the analogous results for the Euler equations.

The connection between the existence of dissipative weak solutions of Euler and the convex integration techniques
used to prove the $h$-principle in geometric problems (and unexpected solutions to differential inclusions) was
first observed in \cite{DS1}. Since then these techniques have been used successfully in other equations of fluid dynamics:
we refer the interested reader to the survey article \cite{DSSurvey}.

\medskip

\subsection{Loss of derivatives and regularization}

Finally, let us make a technical remark. Since the negative power of $\lambda$ in estimate \eqref{e:error} comes from a stationary-phase type argument (Proposition \ref{p:schauder} in Section \ref{s:est1}), the constant $C(v_k,\mathring{R}_k)$ will then depend on higher derivatives of $v_k$ (and of $\mathring{R}_k$). In fact, with $\theta\to \frac{1}{10}$ the number of derivatives $m$ required in the estimates converges to $\infty$. To overcome this {\it loss of derivative} problem, we use the well-known device from the Nash-Moser iteration to mollify $v_k$ and $\mathring{R}_k$ at some appropriate scale $\ell_k$. Although we are chiefly interested in derivative bounds in space, due to the nature of the equation such bounds are connected to derivative bounds in time, necessitating a mollification in space and time. To simplify the presentation we will therefore treat time also as a periodic variable and
we will therefore construct solutions on $\mathbb T^3 \times \mathbb S^1$ rather than on $\mathbb T^3\times [0,1]$.

\medskip

\subsection{Acknowledgements}

The first author acknowledges the support of SFB Grant TR71, the second author acknowledges the support of ERC Grant Agreement No.~277993.

\section{Iteration with double exponential decay}\label{s:iteration}

\subsection{Notation in H\"older norms}

In the following $m=0,1,2,\dots$, $\alpha\in (0,1)$, and $\beta$ is a multiindex. We introduce the usual (spatial) 
H\"older norms as follows.
First of all, the supremum norm is denoted by $\|f\|_0:=\sup_{\T^3}|f|$. We define the H\"older seminorms 
as
\begin{equation*}
\begin{split}
[f]_{m}&=\max_{|\beta|=m}\|D^{\beta}f\|_0\, ,\\
[f]_{m+\alpha} &= \max_{|\beta|=m}\sup_{x\neq y}\frac{|D^{\beta}f(x)-D^{\beta}f(y)|}{|x-y|^{\alpha}}\, .
\end{split}
\end{equation*}
The H\"older norms are then given by
\begin{eqnarray*}
\|f\|_{m}&=&\sum_{j=0}^m[f]_j\\
\|f\|_{m+\alpha}&=&\|f\|_m+[f]_{m+\alpha}.
\end{eqnarray*}

For functions depending on {\em space and time}, we define spatial H\"older norms as
\[
\|v\|_r=\sup_{t}\|v(\cdot, t)\|_r\, ,
\]
whereas the H\"older norms in space {\em and time} will be denoted by $\|\cdot\|_{C^r}$.

\subsection{The iterative scheme} We follow here \cite{DS3} and introduce the Euler-Reynolds system (cp. with Definition 2.1 therein).

\begin{definition}\label{d:euler_reynolds}
Assume $v, p, \mathring{R}$ are $C^1$ functions on $\T^3\times \mathbb S^1$
taking values, respectively, in $\R^3, \R, \S^{3\times 3}_0$. We say that they solve the Euler-Reynolds system if 
\begin{equation}\label{e:euler_reynolds}
\left\{\begin{array}{l}
\partial_t v + \div (v\otimes v) + \nabla p =\div\mathring{R}\\ \\
\div v = 0\, .
\end{array}\right.
\end{equation} 
\end{definition}

The next proposition is the main building block of our construction: the proof of Theorem \ref{t:main}
is achieved by applying it inductively to generate a suitable sequence of solutions to \eqref{e:euler_reynolds}
where the right hand side vanishes in the limit.

\begin{proposition}\label{p:iterate}
Let $e$ be a smooth positive function on $\mathbb S^1$. There exist positive constants $\eta,M$ depending on $e$ with the following property.

Let $\delta \leq 1$ be any positive number and $(v, p, \mathring{R})$  a solution of the Euler-Reynolds
system \eqref{e:euler_reynolds} in $\T^3\times \mathbb S^1$ such that
\begin{equation}\label{e:energy_hyp}
\tfrac{3\delta}{4} e(t) \leq e(t) - \int |v|^2 (x,t)\, dx \leq \tfrac{5\delta}{4} e(t) \qquad \forall t\in \mathbb S^1\, ,
\end{equation}
\begin{equation}\label{e:Reynolds_hyp}
\|\mathring{R}\|_0 \leq \eta \delta\, 
\end{equation}
and
\begin{equation}\label{e:C1}
D := \max \{ 1,\|\mathring{R}\|_{C^1}, \|v\|_{C^1}\}\, .
\end{equation}

For every $\bar\delta \leq \frac{1}{2} \delta^{\frac{3}{2}}$ and every $\eps>0$  
there exists a second triple $(v_1, p_1, \mathring{R}_1)$ which solves
as well the Euler-Reynolds system and satisfies the following estimates:
\begin{equation}\label{e:energy_est}
\tfrac{3\bar{\delta}}{4} e(t) \leq e(t) - \int |v_1|^2 (x,t)\, dx\leq \tfrac{5\bar{\delta}}{4} e(t) \qquad \forall t\in \mathbb S^1\, ,
 \end{equation}
\begin{equation}\label{e:Reynolds_est}
\|\mathring{R}_1\|_{0}\leq \eta\bar{\delta}\, ,
\end{equation}
\begin{equation}\label{e:C^0_est}
\|v_1 - v\|_{0} \leq M \sqrt{\delta}\, ,
\end{equation}
\begin{equation}\label{e:C0_pressure}
\|p_1-p\|_0 \leq M^2 \delta\, ,
\end{equation}
and
\begin{equation}\label{e:C1_est}
\max \{\|v_1\|_{C^1}, \|\mathring{R}_1\|_{C^1}\} \leq A  \delta^{\frac{3}{2}}\left(\frac{D}{\bar{\delta}^{2}}\right)^{1+\eps}\,
\end{equation}
where the constant $A$ depends on $e$, $\eps>0$ and $\|v\|_0$.
\end{proposition}

We next show how to conclude Theorem \ref{t:main} from Proposition \ref{p:iterate}:
the rest of the paper is then devoted to prove the Proposition. 

\begin{proof}[Proofs of Theorem \ref{t:main}]
Let $e$ be as in the statement, i.e. smooth and positive. Without loss of generality 
we can assume that $e$ is defined
on $\R$, with period $2\pi$, and it is smooth and positive on the entire real line.

\medskip

{\bf Step 1.} Fix any arbitrarily small number $\eps>0$ and let $a,b\geq \frac{3}{2}$ be numbers whose choice will be specified later and will depend only on $\eps$. We define $(v_0, p_0, \mathring{R}_0)$
to be identically $0$ and we apply Proposition \ref{p:iterate} inductively with
\[
\delta_n=a^{-b^n}
\]
to produce a sequence $(v_n, p_n, \mathring{R}_n)$ of solutions of the Euler-Reynolds system and numbers $D_n$ satisfying the following requirements:
\begin{equation}\label{e:energy_sequence}
\tfrac{3\delta_n}{4} e(t) \leq e(t) - \int |v_1|^2 (x,t)\, dx\leq \tfrac{5\delta_n}{4} e(t) \qquad \forall t\in 
\mathbb S^1\, ,
\end{equation}
\begin{equation}\label{e:Reynolds_sequence}
\|\mathring{R}_n\|_{0}\leq \eta\delta_n\, ,
\end{equation}
\begin{equation}\label{e:C^0_sequence}
\|v_n - v_{n-1}\|_{0} \leq M \sqrt{\delta_{n-1}}\, ,
\end{equation}
\begin{equation}\label{e:pressure_seq}
\|p_n-p_{n-1}\|_0 \leq M^2 \delta_{n-1}\, .
\end{equation}
\begin{equation}\label{e:C1_sequence}
D_n = \max \{1,\|v_n\|_{C^1}, \|\mathring{R}_n\|_{C^1} \} .
\end{equation}
Observe that with this choice of $\delta_n$ and since $a,b\geq \frac{3}{2}$, $(v_n, p_n)$ converges uniformly to a continuous pair $(v,p)$
and in particular 
\[
\|v_n\|_0\leq M\sum_{j=0}^\infty a^{-\frac{1}{2} b^j}\leq M\sum_{j=0}^\infty \left(\frac{3}{2}\right)^{-\frac{1}{2} \left(\frac{3}{2}\right)^j} .
\]
Therefore, $\|v_n\|_0$ is bounded uniformly, with a constant depending only on $e$.
By Proposition \ref{p:iterate} we have
\[
D_{n+1} \leq A  \delta_n^{\frac{3}{2}} \left(\frac{D_n}{\delta_{n+1}^2}\right)^{1+\eps}\, .
\]
Since $A$ is depending only on $e,\eps$ and $\|v_n\|_0$, which in turn can be estimated in terms of $e$, we can assume that $A$ depends only on $\eps$ and $e$.

We claim that, for a suitable choice of the constants $a,b$ there is a third constant $c>1$ for
which we inductively have the inequality
\[
D_n\leq a^{cb^n}.
\]
Indeed, for $n=0$ this is obvious. 
Assuming the bound for $D_n$, we obtain for $D_{n+1}$
\begin{equation*}
\begin{split}
D_{n+1}&\leq A\frac{a^{-\frac{3}{2}b^n}a^{c(1+\eps) b^n}}{a^{-2(1+\eps)b^{n+1}}}
= Aa^{(-\frac{3}{2}+(1+\eps)(c+2b))b^n}\,.
\end{split}
\end{equation*}
We impose $\eps<\frac{1}{4}$ and set
\[
b=\frac{3}{2}\quad\mbox{and}\quad c=\frac{3(1+2\eps)}{1-2\eps}+\eps.
\]
This choice leads to
\begin{equation*}
cb-\left(-\tfrac{3}{2} + (1+\eps)(c+2b)\right)=\frac{\eps}{2}(1-2\eps)>\frac{\eps}{4}\, .
\end{equation*}
Since $b^n\geq 1$, we conclude
\[
D_{n+1}\leq \left(A a^{-\eps/4}\right) a^{cb^{n+1}}
\]
Choosing $a = A^{\sfrac{4}{\eps}}$ we conclude $D_{n+1}\leq a^{cb^{n+1}}$.

\medskip

{\bf Step 2.} Consider now the sequence $v_n$ provided in the previous step.
By \eqref{e:energy_sequence}, \eqref{e:Reynolds_sequence} and \eqref{e:C^0_sequence}  we conclude that $(v_n, p_n)$ converges uniformly to a solution $(v,p)$
of the Euler equations such that $e(t) =\int |v|^2 (x,t) dx$ for every $t\in \mathbb S^1$.
On the other hand, observe that
\[
\|v_{n+1} - v_n\|_{0} \leq M \sqrt{\delta_n} \leq M a^{-\frac{1}{2} b^n}
\]
and
\[
\|v_{n+1} - v_n\|_{C^1} \leq D_n+D_{n+1} \leq 2 a^{c b^{n+1}}\, .
\]
Therefore
\begin{equation*}
\begin{split}
\|v_{n+1}-v_n\|_{C^\theta} &\leq \|v_{n+1}-v_n\|_0^{1-\theta}\|v_{n+1}-v_n\|_{C^1}^\theta\\
&\leq 2M a^{\left(\theta cb - \frac{(1-\theta)}{2}\right) b^n}\, .
\end{split}
\end{equation*}
If 
\[
\theta < \frac{1}{1+2 cb}=\frac{1-2\eps}{10+19\eps-6\eps^2},
\] 
then $\theta cb - \frac{(1-\theta)}{2}<0$ and therefore
$\{v_n\}$ is a Cauchy sequence on $C^\theta$,
which implies that it converges in the $C^\theta$ norm. 

\medskip

We have shown that, for every $\eps<\frac{1}{4}$ and every $\theta < \frac{1-2\eps}{10+19\eps-6\eps^2}$ there is a pair $(v,p)\in C^\theta (\T^3\times \mathbb S^1, \R^3)\times C
(\T^3\times \mathbb S^1)$ as in Theorem \ref{t:main}. Letting $\eps\downarrow 0$ we obtain the conclusions
of Theorem \ref{t:main} (and indeed even the H\"older regularity in time).
\end{proof}

\begin{remark}\label{r:Physica_Z}
Using the bounds on $\delta_n$ and $D_n$ in the proof above, we can obtain an estimate on the energy spectrum of $v$.
First of all we observe (cp.~Section \ref{s:setup_for_prop}) that in Fourier space 
$v_{n+1}-v_n$ is essentially supported in a frequency band around wavenumber $\lambda_n$. For $\lambda_n$ we then have the relation
$$
\|v_{n+1}-v_n\|_{C^1}\sim \|v_{n+1}-v_n\|_{C^0}\,\lambda_n.
$$
Therefore, Step 2 of the proof above implies
$$
\lambda_n\sim a^{(bc+\tfrac{1}{2})b^n},
$$
and consequently the energy spectrum satisfies
$$
E(\lambda_n)\sim \frac{\delta_n}{\lambda_n}\sim a^{-(\tfrac{3}{2}+bc)b^n}\sim \lambda_n^{-\frac{3+2bc}{1+2bc}}\,.
$$
Plugging in the choice of $b,c$ from Step 1 of the proof yields in the limit $\eps\to 0$
$$
E(\lambda_n)\sim \lambda_n^{-\sfrac65}.
$$
\end{remark}

\subsection{Plan of the remaining sections} Except for Section \ref{s:final}, in which we prove the side Remark \ref{r:time_and_pressure},
the remaining sections are all devoted to the proof of Proposition \ref{p:iterate}. 

Section \ref{s:setup_for_prop}
contains the precise definition of the maps $(v_1, p_1, \mathring{R}_1)$ 
of Proposition \ref{p:iterate}. The maps will depend upon various parameters, which will
be specified only at the end.

Section \ref{s:est1} contains some preliminaries on classical estimates for the H\"older norms of products and compositions
of functions,
some classical Schauder estimates for the elliptic operators involved in the construction
and a "stationary phase lemma" (Proposition \ref{p:schauder}) for the H\"older norms of highly oscillatory functions. This last lemma is also a quite classical fact, but it plays a key role in our estimates.

In Section \ref{s:est2} we prove the key estimates on the main building blocks of the construction in terms
of the relevant parameters: all these estimates are collected in the technical
Proposition \ref{p:W}.

The various tools introduced in the Sections \ref{s:est1} and \ref{s:est2} are then used in Section \ref{s:est3},
\ref{s:est4} and \ref{s:est5} to derive the fundamental estimates on the H\"older norms of $v_1$
and $\mathring{R}_1$ in terms of the relevant parameters. In particular:
\begin{itemize}
\item Section \ref{s:est3} contains the estimates on $v_1$;
\item Section \ref{s:est4} the estimate on the kinetic energy $\int |v_1|^2$; 
\item Section \ref{s:est5}
the estimates on the Reynolds stress $\mathring{R}_1$.
\end{itemize}

Finally, in Section \ref{s:iterate_proof} the estimates of the Sections \ref{s:est3}, \ref{s:est4} and \ref{s:est5} are used to tune the parameters and prove Proposition \ref{p:iterate}.

\section{Definition of the maps $v_1, p_1$ and $\mathring{R}_1$}\label{s:setup_for_prop}

From now on we fix a triple $(v, p, \mathring{R})$ and numbers $\delta, \bar{\delta}, \eps>0$ as in Proposition \ref{p:iterate}. As in \cite{DS3} the new velocity $v_1$ is obtained by 
adding two perturbations, $w_o$ and $w_c$:
\begin{equation}\label{e:def_v_1}
v_1=v+w_o+w_c = v_1+w\, ,
\end{equation}
where $w_c$ is a {\em corrector} to ensure that $v_1$ is divergence-free. Thus, 
$w_c$ is defined as
\begin{equation}\label{e:Leray_projection}
w_c:= - \Q w_o\,
\end{equation}
where $\Q = {\rm Id}-\P$ and $\P$ is the Leray projection operator, see \cite[Definition 4.1]{DS3}.

\subsection{Conditions on the parameters} The main perturbation $w_o$ is a highly oscillatory function which depends on three parameters: a (small) length scale $\ell>0$ and (large) frequencies $\mu,\lambda$ such that
\[
\lambda,\,\mu,\, \frac{\lambda}{\mu}\in\N.
\] 
In the subsequent sections we will assume the following inequalities:
\begin{equation}\label{e:range}
\mu\geq \delta^{-1}\geq 1,\quad
 \ell^{-1}\geq \frac{D}{\eta\delta}\geq 1,\quad
\lambda\geq \max \left\{(\mu D)^{1+\omega}, \ell^{-(1+\omega)}\right\}\, .
\end{equation}
Here $\omega:=\frac{\eps}{2+\eps}>0$ so that
\[
1+\eps = \frac{1+\omega}{1-\omega}\,.
\]
Of course, at the very end, the proof of Proposition \ref{p:iterate} will use a specific choice of the
parameters, which will be shown to respect the above conditions. However, at this stage
the choices in \eqref{e:range} seem rather arbitrary. We could leave the parameters completely free and
carry all the relevant estimates in general, but this would give much more complicated and lengthy formulas
in all of them.  
It turns out that the conditions \eqref{e:range} above greatly simplifies many
computations.
  
\subsection{Definition of $w_o$}
In order to define $w_o$ we draw heavily upon the techniques introduced in \cite{DS3}.
\begin{itemize}
\item First of all we let $r_0>0$, $N, \lambda_0\in \mathbb N$, $\Lambda_j\subset \{k\in \mathbb Z^3 : |k|=\lambda_0\}$
and $\gamma_k^{(j)} \in C^\infty (B_{r_0} ({\rm Id}))$ be as in \cite[Lemma 3.2]{DS3}.
\item Next we let $\mathcal{C}_j\subset \mathbb Z^3$, $j\in \{1, \ldots, 8\}$ and the functions $\alpha_k$ be as in \cite[Section 4.1]{DS3};
as in that section, we define the functions
\[
\phi^{(j)}_{k,\mu} (v, \tau) := \sum_{l\in \mathcal{C}_j} \alpha_l (\mu v) e^{-i \frac{k\cdot l}{\mu} \tau}\, .
\]
\end{itemize}
Next, we let $\chi\in C^\infty_c (\R^3\times \R)$ be a smooth standard nonnegative radial kernel supported in $[-1,1]^4$ and 
we denote by
\[
\chi_\ell (x,t) := \frac{1}{\ell^4} \chi \left(\frac{x}{\ell}, \frac{t}{\ell}\right)
\] 
the corresponding family of mollifiers. We define
\begin{align*}
v_{\ell} (x,t) &=\int_{\T^3\times \mathbb S^1}v(x-y,t-s)\chi_{\ell}(y,s)\,dy\, ds\\
\mathring{R}_{\ell} (x,t) &=\int_{\T^3\times \mathbb S^1}\mathring{R}(x-y,t-s)\chi_{\ell}(y,s)\,dy\, ds.
\end{align*}

Similarly to \cite[Section 4.1]{DS3}, we define the function 
\begin{equation}\label{e:def_rho}
\rho_{\ell} (t) := \frac{1}{3(2\pi)^3} \left(e(t) (1-\bar\delta) - \int_{\T^3} |v_\ell|^2 (x,t)\, dx \right)\, 
\end{equation}
and the symmetric $3\times 3$ matrix field
\begin{equation}\label{e:def_R}
R_{\ell} (x,t) = \rho_{\ell} (t) {\rm Id} - \mathring{R}_\ell (x,t)\, .
\end{equation}
Finally, $w_o$ is defined by
\begin{equation}\label{e:def_w_o}
w_o (x,t) := \sqrt{\rho_{\ell}(t)}\sum_{j=1}^8\sum_{k\in\Lambda_j}\gamma_k^{(j)}\left(\frac{R_{\ell}(x,t)}{\rho_{\ell}(t)}\right)\phi_{k,\mu}^{(j)}\left(v_\ell (x,t),\lambda t\right)B_ke^{i\lambda k\cdot x}\, ,
\end{equation}
where $B_k\in \mathbb C^3$ are vectors of unit length satisfying the assumptions of \cite[Proposition 3.1]{DS3}. Recall that the maps $\gamma^{(j)}_k$ are defined only in $B_{r_0}({\rm Id})$. The function $w_o$ is nonetheless well defined: the fact that the arguments of $\gamma^{(j)}_k$ are contained in $B_{r_0}({\rm Id})$ will be ensured by the choice of $\eta$ in Section \ref{s:constants} below.

\subsection{The constants $\eta$ and $M$}\label{s:constants}

We start by observing that, by standard estimates on convolutions
\begin{align}
\|v_\ell\|_r+\|\mathring{R}_\ell\|_r&\leq C (r) D\ell^{-r}\quad\textrm{for any $r\geq 1$, }\label{e:conv1_stand}\\
\|v_\ell-v\|_0+\|\mathring{R}_\ell-\mathring{R}\|_0&\leq CD\ell\,\, ,\label{e:conv0_stand}
\end{align}
where the first constant depends only on $r$ and the second is universal.
By writing $\bigl||v_\ell|^2-|v|^2\bigr|\leq |v-v_\ell|^2+2|v||v-v_\ell|$ we deduce
\begin{align}
\int_{\T^3}\bigl||v_\ell|^2-|v|^2\bigr|\,dx&\leq C(D\ell)^2+Ce(t)^{1/2}D\ell\label{e:L2norm_est}\\
&\leq C\eta\delta\left(\max_{t}e(t)^{1/2}+1\right),\label{e:etadelta}
\end{align}
where the last inequality follows from \eqref{e:range}. This leads to the following lower bound on $\rho_\ell$:
\begin{align}
\rho_\ell(t)&\geq \frac{1}{3(2\pi)^3}\left(e(t)\left(1-\frac{\delta}{2}\right)-\int_{\T^3}|v|^2\,dx-\int_{\T^3}\bigl||v_\ell|^2-|v|^2\bigr|\,dx\right)\nonumber\\
&\stackrel{\eqref{e:etadelta}}{\geq}  \frac{1}{3(2\pi)^3}\left(\frac{\delta}{4}\min_{t}e(t)-C\eta\delta\left(\max_{t}e(t)^{1/2}+1\right)\right)\label{e:rho_below}
\end{align}
We then choose $0<\eta<1$ so that the quantity on the right hand side is greater than $\frac{2\eta\delta}{r_0}$. This is clearly possible with a choice of $\eta$ only depending on $e$. In turn, this leads to 
\begin{equation}
\left\|\frac{R_\ell}{\rho_\ell} - {\rm Id}\right\|_0 \leq \frac{\|\mathring{R}_\ell\|_0}{\min_t\rho_\ell(t)}\leq \frac{r_0}{2}.\label{e:r_0/2}
\end{equation}
Therefore $w_o$ in \eqref{e:def_w_o} is well-defined.

\medskip

In an analogous way we estimate $\rho_\ell$ from above as
\begin{align}
\rho_\ell(t)&\leq \frac{1}{3(2\pi)^3}\left(e(t)-\int_{\T^3}|v|^2\,dx+\int_{\T^3}\bigl||v_\ell|^2-|v|^2\bigr|\,dx\right)\nonumber\\
&\leq  \frac{1}{3(2\pi)^3}\left(\frac{5\delta}{4}\max_{t}e(t)+C\delta\left(\max_{t}e(t)^{1/2}+1\right)\right)\nonumber\\
&\leq C\delta\left(1+\max_te(t)\right)\, . \label{e:rho_above}
\end{align}
Since $|w_o|$ can be estimated as
\[
|w_o(x,t)|\leq C\sqrt{\rho_\ell(t)}\,,
\]
we can choose the constant $M$, depending only on $e$, in such a way that
\begin{equation}\label{e:w_o_C0}
\|w_o\|_0\leq \frac{M}{2}\sqrt{\delta}\,.
\end{equation}
This is essentially the major point in the definition of $M$: the remaining terms leading
to \eqref{e:C^0_est} and \eqref{e:C0_pressure} will be shown to be negligible thanks to an appropriate
choice of the parameters $\lambda, \mu$ and $\ell$. We will therefore require that, in addition to
\eqref{e:w_o_C0}, $M\geq 1$

\subsection{The pressure $p_1$} The pressure $p_1$ differs slightly from the corresponding one chosen
in \cite{DS3}. It is given by
\begin{equation}\label{e:def_p_1}
p_1=p- \frac{|w_o|^2}{2} - \frac{2}{3} \langle v-v_\ell, w\rangle\,.
\end{equation}
Observe that, by \eqref{e:w_o_C0}, we have
\begin{equation}\label{e:pressure_M}
\|p_1-p\|_0 \leq \frac{M^2}{4} \delta + \|v-v_\ell\|_0 \|w\|_0\, .
\end{equation}

\subsection{The Reynolds stress $\mathring{R}_1$} The Reynolds stress $\mathring{R}_1$ is defined by a slightly more
complicated formula than the corresponding one in \cite[Section 4.5]{DS3}. Recalling the operator $\RR$ from 
\cite[Definition 4.2]{DS3} we define $\mathring{R}_1$ as 
\begin{equation}\label{e:def_R_1}
\begin{split}
\mathring{R}_1 &= \RR [ \partial_t w + {\rm div} (w\otimes v_\ell + v_\ell \otimes w)]\\
&+ \RR [{\rm div} (w\otimes w + \mathring{R}_\ell - \textstyle{\frac{|w_o|^2}{2}}{\rm Id})]\\
&+ [w\otimes (v - v_\ell) + (v-v_\ell)\otimes w
 - \textstyle{\frac{2 \langle (v-v_{\ell}), w\rangle}{3}} {\rm Id}]\\
&+[\mathring{R}_{\ell} - \mathring{R}]\, .
\end{split}
\end{equation}
The summands in the third and fourth line
are obviously trace-free and symmetric. The summands in the first and second line are symmetric and
trace-free because of the properties of the operator $\mathcal{R}$ (cp. with \cite[Lemma 4.3]{DS3}).
Moreover, the expressions to which the operator $\mathcal{R}$ is applied have average $0$.
For the second line this is obvious because the expression is the divergence of a matrix field. 
As for the first line, since $w= \P w_o$, its average is zero by the definition of the operator $\P$.
Therefore the average of $\partial_t w$ is also zero. 
The remaining term is a divergence and hence its average equals $0$.

We now check that the triple $(v_1, p_1, \mathring{R}_1)$ satisfies the Euler-Reynolds system. 
First of all, recall that $\nabla g = {\rm div} (g {\rm Id})$ for any smooth function
$g$ and that ${\rm div} \RR F = F$ for any smooth $F$ with average $0$. Since we already observed that
the expressions to which $\mathcal{R}$ is applied average to $0$, we can compute
\[
{\rm div}\, \mathring{R}_1 - \nabla p_1 = \partial_t w + {\rm div} (w\otimes w) + {\rm div} (w\otimes v + v\otimes w)
- \nabla p + {\rm div}\, \mathring{R}\, .
\]
But recalling that $\div \mathring{R} = \partial_t v + {\rm div}\, (v\otimes v) + \nabla p$ we also get
\[
{\rm div}\, \mathring{R}_1 - \nabla p_1 = \partial_t (v+w) + {\rm div}\, [w\otimes w + v\otimes v + w\otimes v + v\otimes w]\, .
\]
Since $v_1=v+w$ we then conclude the desired identity.

In order to complete the
proof of Proposition \ref{p:iterate} we need to show that the (minor) estimates \eqref{e:C^0_est}, \eqref{e:C0_pressure} and 
the (major) estimates \eqref{e:energy_est}, \eqref{e:Reynolds_est}, \eqref{e:C1_est} hold: 
essentially all the rest of the paper is devoted to prove them.

\subsection{Constants in the estimates} The rest of the paper is devoted to estimating several H\"older norms of the various
functions defined so far. The constants appearing in the estimates will always be denoted by the letter $C$, which might be
followed by an appropriate subscript. First of all, 
by this notation we will throughout understand that the value may change from line to line.
In order to keep track of the quantities on which these constants depend, we will use subscripts to
make the following distinctions. 

\begin{itemize}
\item $C$: without a subscript will denote universal constants;
\item $C_h$: will denote constants in estimates concerning standard functional inequalities in H\"older spaces $C^r$ (such as \eqref{e:Holderinterpolation}, \eqref{e:Holderproduct}). These constants depend only on the specific norm used and therefore only
on the parameter $r\geq 0$: however we keep track of this dependence because the number $r$ will be chosen only at the end
of the proof of Proposition \ref{p:iterate} and its value may be very large;
\item $C_e$: throughout the rest paper the prescribed energy density $e=e(t)$ of Theorem \ref{t:main} and Proposition
\ref{p:iterate} will be assumed to be a fixed smooth function bounded below and above by positive constants; several estimates depend on these bounds and the related constants will be denoted by $C_e$;
\item $C_v$: in addition to the dependence on $e$, there will be estimates which depend also on the supremum norm 
of the velocity field $\|v\|_0$ (this explains
the origin of the constant $A$ in \eqref{e:C1_est}); 
\item $C_s$, $C_{e,s}$, $C_{v,s}$: will denote constants which are typically involved in Schauder estimates for $C^{m+\alpha}$ norms of elliptic operators,
when $m\in\mathbb N$ and $\alpha\in ]0,1[$; these constants not only depend on the specific norm used, but they also degenerate as $\alpha\downarrow 0$ and $\alpha\uparrow 1$; the ones denoted by $C_{e,s}$ and $C_{v,s}$
depend also, respectively, upon $e$ and upon $e$ and $\|v\|_0$.
\end{itemize}

Observe in any case that, no matter which subscript is used, such constants {\em never} depend on the parameters $\mu, \ell, \delta, \lambda$ and $D$; they are, however, allowed to depend on $\omega$ and $\eps$.

\section{Preliminary H\"older estimates}\label{s:est1}

In this section we collect several estimates which will be used throughout the rest
of the paper. 

We start with the following elementary inequalities:
\begin{equation}\label{e:Holderinterpolation}
[f]_{s}\leq C_h\bigl(\eps^{r-s}[f]_{r}+\eps^{-s}\|f\|_0\bigr)
\end{equation}
for $r\geq s\geq 0$ and $\eps>0$, and 
\begin{equation}\label{e:Holderproduct}
[fg]_{r}\leq C_h\bigl([f]_r\|g\|_0+\|f\|_0[g]_r\bigr)
\end{equation}
for any $1\geq r\geq 0$, where the constants depend only on $r$ and $s$. 
From \eqref{e:Holderinterpolation} with $\eps=\|f\|_0^{\sfrac1r}[f]_r^{-\sfrac1r}$ we obtain the 
standard interpolation inequalities
\begin{equation}\label{e:Holderinterpolation2}
[f]_{s}\leq C_h\|f\|_0^{1-\sfrac{s}{r}}[f]_{r}^{\sfrac{s}{r}}.
\end{equation}
Next we collect two classical estimates on the H\"older norms of compositions. These are also standard, for instance
in applications of the Nash-Moser iteration technique. For the convenience of the reader we recall the short proof.

\begin{proposition}\label{p:chain}
Let $\Psi: \Omega \to \mathbb R$ and $u: \R^n \to \Omega$ be two smooth functions, with $\Omega\subset \R^N$. 
Then, for every $m\in \mathbb N \setminus \{0\}$ there is a constant $C_h$ (depending only on $m$,
$N$ and $n$) such that
\begin{align}
\left[\Psi\circ u\right]_m &\leq C_h \sum_{i=1}^m [\Psi]_i \|u\|_0^{i-1} [u]_m\label{e:chain0}\\
\left[\Psi\circ u\right]_m &\leq C_h \sum_{i=1}^m [\Psi]_i [u]_1^{(i-1)\frac{m}{m-1}} [u]_m^{\frac{m-i}{m-1}}\, .
\label{e:chain1}
\end{align} 
\end{proposition}

\begin{proof}
Denoting by $D^j$ any partial derivative of order $j$, the chain rule can be written symbolically as
\begin{equation}\label{e:chainrule}
D^m(\Psi\circ u)=\sum_{l=1}^m(D^l\Psi)\circ u\sum_{\sigma}C_{l,\sigma}(Du)^{\sigma_1}(D^2u)^{\sigma_2}\dots(D^mu)^{\sigma_m}
\end{equation}
for some constants $C_{l,\sigma}$,
where the inner sum is over $\sigma=(\sigma_1,\dots,\sigma_m)\in\N^m$ such that
\begin{equation*}
\sum_{j=1}^m\sigma_j=l,\quad \sum_{j=1}^mj\sigma_j=m.
\end{equation*}
From \eqref{e:Holderinterpolation2} we have
\begin{enumerate}
\item[(a)]  $[u]_j\leq C_h\|u\|_0^{1-\frac{j}{m}}[u]_m^{\frac{j}{m}}$ for $j\geq 0$;
\item[(b)] $[u]_j\leq C_h[u]_1^{1-\frac{j-1}{m-1}}[u]_m^{\frac{j-1}{m-1}}$ for $j\geq 1$.
\end{enumerate}
Then \eqref{e:chain0} and \eqref{e:chain1} follow from applying (a) and (b) to \eqref{e:chainrule}, respectively.
\end{proof}

\subsection{Estimates on $\phi^{(j)}_{k,\mu}$} 
Recall that $\phi^{(j)}_{k,\mu}=\phi^{(j)}_{k,\mu}(v,\tau)$ are defined on
$\R^3\times \mathbb S^1$ and they are smooth (here $v$ is treated as an independent variable). Because the $\tau$-derivatives are not bounded in $v$,
we introduce the seminorms 
$$
[\cdot]_{r,R}=[\cdot]_{C^r(B_R(0)\times\mathbb S^1)}
$$ 
to denote the H\"older seminorms of the restriction of the corresponding function on the set $B_R (0)\times \mathbb S^1$.

\begin{proposition}\label{p:phi}
There are constants $C_h$ depending only on $m\in \mathbb N$ and 
such that the following estimates hold:
\begin{align}
\left[\phi^{(j)}_{k,\mu}\right]_{m, R} + R^{-1} \left[\partial_\tau \phi^{(j)}_{k,\mu}\right]_{m, R}
+R^{-2} \left[\partial_{\tau\tau} \phi^{(j)}_{k,\mu}\right]_{m,R} &\leq C_h \mu^m \label{e:phi}\\
\left[\partial_\tau \phi^{(j)}_{k,\mu} + i(k\cdot v)\phi_{k,\mu}^{(j)}\right]_m &\leq C_h \mu^{m-1} \label{e:transportphi}\\
R^{-1} \left[\partial_\tau \left(\partial_\tau \phi^{(j)}_{k,\mu} + i(k\cdot v)\phi_{k,\mu}^{(j)}\right)\right]_{m,R} &\leq C_h \mu^{m-1}\label{e:transportphi_t}
\end{align}
\end{proposition}
\begin{proof} 
We recall briefly the definition of the maps $\phi^{(j)}_{k,\mu}$ from
\cite[Section 4.1]{DS3}. 
First of all we fix two constants $c_1$ and $c_2$ such that $\frac{\sqrt{3}}{2} < c_1 < c_2 < 1$
and then $\varphi\in C^\infty_c (B_{c _2} (0))$ which is nonnegative and identically
$1$ on the ball $B_{c_1} (0)$. 
We then set
\[
\psi (v) :=\sum_{k\in \Z^3} (\varphi (v-k))^2 \qquad \mbox{and}
\qquad 
\alpha_k(v) := \frac{\varphi (v -k)}{\sqrt{\psi(v)}}\, .
\]
By the choice of $c_1$ we easily conclude that $\psi^{-\frac{1}{2}}\in C^\infty$.
On the other hand it is also obvious that $\psi (v-k) = \psi (v)$. Thus there is a function
$\alpha\in C^\infty_c (B_1 (0))$ such that $\alpha_k (v) = \alpha (v-k)$.

We next consider the lattice $\Z^3\subset \R^3$ and its quotient by $(2\Z)^3$ and we denote 
by $\mathcal{C}_j$ , $j=1, \ldots, 8$ the 8 equivalence classes of $\Z^3/\sim$. Finally,
in \cite[Section 4.1]{DS3} we set
\begin{equation}\label{e:def_phi}
\phi^{(j)}_k(v,\tau) := \sum_{l\in\mathcal{C}_j}\alpha_l(\mu v)e^{-i(k\cdot \frac{l}{\mu})\tau}.
\end{equation}
Observe that the functions $\{\alpha_l : l\in \mathcal{C}_j\}$
have pairwise disjoint supports. Therefore the estimate
\[
\left[\phi^{(j)}_{k,\mu}\right]_m \leq C [\alpha]_m \mu^m \leq C_h \mu^m\, 
\]
follows trivially. Next, 
\[
\partial_\tau \phi^{(j)}_k(v,\tau) := \sum_{l\in\mathcal{C}_j} - i \left(k\cdot \frac{l}{\mu}\right) \alpha_l(\mu v)e^{-i(k\cdot \frac{l}{\mu})\tau}\, .
\]
On the other hand, if $|v|\leq R$, then $\alpha_l (\mu v) =0$ for any $l$ with $|l|\geq \mu R +2$: hence
\[
\left[\partial_\tau \phi^{(j)}_k\right]_{m, R} \leq |k| \left(R + 2\mu^{-1}\right)
 [\varphi]_m \mu^m \leq C_h R\, \mu^m\, 
\]
(in principle the constant $C_h$ depends on $k$, but on the other hand $k$ ranges in $\cup_j \Lambda_j$, which
is a finite set).
A similar argument applies to $\partial_{\tau\tau} \phi^{(j)}_{k, \mu}$ and hence concludes
the proof of \eqref{e:phi}.

\medskip

We finally compute
\begin{align*}
D^m_v \left(\partial_\tau \phi^{(j)}_{k,\mu} + i(k\cdot v)\phi_{k,\mu}^{(j)}\right)
&= \sum_{l\in\mathcal{C}_j} i k\cdot \left(v-\frac{l}{\mu}\right) \mu^m [D^m\alpha] (\mu (v-l))e^{-i(k\cdot \frac{l}{\mu})\tau}\\
&\quad + \mu^{m-1} \sum_{l\in\mathcal{C}_j} i k\otimes [D^{m-1}\alpha] (\mu (v-l))e^{-i(k\cdot \frac{l}{\mu})\tau}\, .
\end{align*}
Recall however that $\alpha\in C^\infty_c (B_1 (0))$: thus $|v-\frac{l}{\mu}|\leq \mu^{-1}$ if
$[D^m \alpha] (\mu (v-l))\neq 0$. It follows easily that
\[
\left[\partial_\tau \phi^{(j)}_{k,\mu} + i(k\cdot v)\phi_{k,\mu}^{(j)}\right]_m
\leq C \mu^{m-1} \left([\alpha]_m + |k| [\alpha]_{m-1}\right) \leq C_h \mu^{m-1}\, ,
\]
which proves \eqref{e:transportphi}. On the other hand, differentiating once more the identities
in $\tau$, \eqref{e:transportphi_t} follows from the same arguments used above for $[\partial_\tau \phi]_{m,R}$.
\end{proof}

\subsection{Schauder estimates for elliptic operators}
We now recall some classical Schauder estimates for the various operators involved in the
construction. These estimates were already collected in \cite[Proposition 5.1]{DS3} and will
be used several times in what follows. We state them again for the readers convenience and
because of the convention on constants as set in Section \ref{s:constants}, and refer to 
\cite[Definitions 4.1, 4.2]{DS3} for the precise definition of the operators $\mathcal{P}$, $\mathcal{Q}$ and
$\mathcal{R}$.

\begin{proposition}\label{p:GT}
For any $\alpha\in (0,1)$ and any $m\in \N$ there exists a constant $C_s (m, \alpha)$ so that the following estimates hold:
\begin{eqnarray}
&&\|\mathcal{Q} v\|_{m+\alpha} \leq C_s (m,\alpha) \|v\|_{m+\alpha}\label{e:Schauder_Q}\\
&&\|\mathcal{P} v\|_{m+\alpha} \leq C_s (m,\alpha) \|v\|_{m+\alpha}\label{e:Schauder_P}\\
&&\|\mathcal{R} v\|_{m+1+\alpha} \leq C_s (m,\alpha) \|v\|_{m+\alpha}\label{e:Schauder_R}\\
&&\|\mathcal{R} ({\rm div}\, A)\|_{m+\alpha}\leq C_s (m,\alpha) \|A\|_{m+\alpha}\label{e:Schauder_Rdiv}\\
&&\|\mathcal{R} \mathcal{Q} ({\rm div}\, A)\|_{m+\alpha}\leq C_s (m,\alpha) \|A\|_{m+\alpha}\label{e:Schauder_RQdiv}\, .
\end{eqnarray}
\end{proposition}

\subsection{Stationary phase lemma}
Finally, we state a key ingredient of our construction, which yields estimates for
highly oscillatory functions. Though this proposition is also essentially contained in \cite{DS3}, it is nowhere
explicitly stated in this form. Since it will be used several more times and in a more subtle way than in \cite{DS3},
it is useful to isolate it from the rest. 

\begin{proposition}\label{p:schauder}
Let $k\in\Z^3\setminus\{0\}$ and $\lambda\geq 1$. 

(i) For any $a\in C^{\infty}(\T^3)$ and $m\in\N$ we have
\begin{equation}\label{e:average}
\left|\int_{\T^3}a(x)e^{i\lambda k\cdot x}\,dx\right|\leq \frac{[a]_m}{\lambda^m}.
\end{equation}

(ii) Let $k\in \mathbb Z^3\setminus \{0\}$. For a smooth vector field $F\in C^{\infty}(\T^3;\R^3)$ let 
$F_{\lambda}(x):=F(x)e^{i\lambda k\cdot x}$. Then we have
\begin{align*}
\|\RR(F_\lambda)\|_{\alpha}&\leq \frac{C_s}{\lambda^{1-\alpha}}\|F\|_0+\frac{C_s}{\lambda^{m-\alpha}}[F]_m+\frac{C_s}{\lambda^m}[F]_{m+\alpha},\\
\|\RR \mathcal{Q} (F_\lambda)\|_{\alpha}&\leq \frac{C_s}{\lambda^{1-\alpha}}\|F\|_0+\frac{C_s}{\lambda^{m-\alpha}}[F]_m+\frac{C_s}{\lambda^m}[F]_{m+\alpha},
\end{align*}
where $C_s=C_s(m,\alpha)$ (i.e. the constant does not depend on $\lambda$ nor on $k$).
\end{proposition}

\begin{proof}
For $j=0,1,\dots$ define
\begin{align*}
A_j(y,\xi)&:=-i\left[\frac{k}{|k|^2}\left(i\frac{k}{|k|^2}\cdot\nabla\right)^ja(y)\right]e^{ik\cdot \xi}\, ,\\
B_j(y,\xi)&:=\left[\left(i\frac{k}{|k|^2}\cdot\nabla\right)^ja(y)\right]e^{ik\cdot \xi}\, .
\end{align*}
Direct calculation shows that
\[
B_j(x,\lambda x)=\frac{1}{\lambda}\div\bigl[A_j(x,\lambda x)\bigr]+\frac{1}{\lambda}B_{j+1}(x,\lambda x).
\]
In particular, for any $m\in\N$
\begin{equation*}
a(x)e^{i\lambda k\cdot x}=B_0(x,\lambda x)=\frac{1}{\lambda}\sum_{j=0}^{m-1} \frac{1}{\lambda^j}\div\bigl[A_j(x,\lambda x)\bigr]+\frac{1}{\lambda^{m}}B_m(x,\lambda x)
\end{equation*}
Integrating this over $\T^3$ and using that $|k|\geq 1$ we obtain \eqref{e:average}.

Next, using \eqref{e:Holderinterpolation} and \eqref{e:Holderproduct} we conclude
\begin{equation*}
\begin{split}
\|A_j(\cdot,\lambda\cdot)\|_{\alpha}&\leq C\left(\lambda^{\alpha}[a]_j+[a]_{j+\alpha}\right)\\
&\leq C\lambda^{j+\alpha}\left(\lambda^{-m}[a]_m+\|a\|_0\right) \qquad \mbox{for any $j\leq m-1$}
\end{split}
\end{equation*}
and similarly
\begin{equation*}
\|B_m(\cdot,\lambda\cdot)\|_{\alpha}\leq C\left(\lambda^{\alpha}[a]_m+[a]_{m+\alpha}\right)\, .
\end{equation*}
Applying the previous computations to each component of the vector field  $F$ we then get the identity
\begin{equation*}
F(x)e^{i\lambda k\cdot x}=G_0(x,\lambda x)=\frac{1}{\lambda}\sum_{j=0}^{m-1} \frac{1}{\lambda^j}\div\bigl[H_j(x,\lambda x)\bigr]+\frac{1}{\lambda^{m}} G_m(x,\lambda x)
\end{equation*}
where the $H_j$ are matrix-valued functions (not necessarily symmetric) and $G_m$ is a vector
field. $H_j$ and $G_m$ enjoy the same estimates of $A_j$ and $B_m$ respectively. Thus, using \eqref{e:Schauder_R},
\eqref{e:Schauder_Rdiv} and \eqref{e:average} we conclude
\begin{align*}
\|\RR (F_\lambda)\|_{\alpha}&\leq C_s \left(\frac{1}{\lambda}\sum_{j=0}^{m-1} \frac{1}{\lambda^j}\|H_j(\cdot,\lambda \cdot)\|_{\alpha}+\frac{1}{\lambda^{m}}\|G_m(\cdot,\lambda \cdot)\|_{\alpha}\right)\\
&\leq C_s\left( \frac{1}{\lambda^{1-\alpha}}\|F\|_0+\frac{1}{\lambda^{m-\alpha}}[F]_m+\frac{1}{\lambda^m}[F]_{m+\alpha}\right)\,.
\end{align*}
Finally, using \eqref{e:Schauder_Q}, \eqref{e:Schauder_R} and \eqref{e:Schauder_RQdiv} we get
\[
\|\RR \mathcal{Q} (F_\lambda)\|_{\alpha}\leq C_s\left(\frac{1}{\lambda^{1-\alpha}}\|F\|_0+\frac{1}{\lambda^{m-\alpha}}[F]_m+\frac{1}{\lambda^m}[F]_{m+\alpha}\right)
\]
as well.
\end{proof}

\section{Doubling the variables and corresponding estimates}\label{s:est2}

It will be convenient to write $w_o$ as
\[
w_o(x,t)=W(x,t,\lambda t,\lambda x),
\]
where
\begin{align}
W(y,s,\tau,\xi)&:=\sum_{|k|=\lambda_0}a_k(y,s,\tau)B_ke^{ik\cdot \xi}\label{e:bigW}\\
&=\sqrt{\rho_{\ell}(s)}\sum_{j=1}^8\sum_{k\in\Lambda_j}\gamma_k^{(j)}\left(\frac{R_\ell(y,s)}{\rho_\ell(s)}\right)\phi_{k,\mu}^{(j)}\left(v_{\ell} (y,s),\tau\right)B_ke^{ik\cdot \xi}
\end{align}
(cp. with \cite[Section 6]{DS3}). The following Proposition corresponds to \cite[Proposition 6.1]{DS3}, with an important difference:
the estimates stated here keep not only track of the dependence of the constants on the parameter $\mu$, but also on the parameter
$\ell$ and the functions $v$ and $\mathring{R}$ (as it can be easily observed, these estimates do not depend on $p$): more precisely
we will make explicit their dependence on $\delta$ and $D$ (for the constants recall the convention
stated in Section \ref{s:constants}). Observe that
all the estimates claimed below are in {\em space only}!

\begin{proposition}\label{p:W}
(i) Let $a_k\in C^{\infty}(\T^3\times\mathbb S^1\times\R)$ be given by \eqref{e:bigW}. Then for any $r\geq 1$  
and any $\alpha\in [0,1]$ we have the following estimates:
\begin{eqnarray}
\|a_k(\cdot,s,\tau)\|_r &\leq & C_e \sqrt{\delta} \left(\mu^rD^r+\mu D\ell^{1-r}\right)\label{e:a,r}\\
\|\partial_\tau a_k(\cdot,s,\tau)\|_r + \|\partial_{\tau\tau} a_k (\cdot, s, \tau)\|_r&\leq & C_v \sqrt{\delta} \left(\mu^rD^r+\mu D\ell^{1-r}\right)\label{e:a_t,r}\\
\|(\partial_\tau a_k+i(k\cdot v_{\ell})a_k)(\cdot,s,\tau)\|_r&\leq & C_e \sqrt{\delta} \left(\mu^{r-1}D^r+ D\ell^{1-r}\right)\label{e:a_trans,r}\\
\|\partial_\tau(\partial_\tau a_k+i(k\cdot v_{\ell})a_k)(\cdot,s,\tau)\|_r&\leq & C_v \sqrt{\delta} \left(\mu^{r-1}D^r+ D\ell^{1-r}\right)\label{e:a_trans_t,r}
\end{eqnarray}
\begin{eqnarray}
\|a_k(\cdot,s,\tau)\|_\alpha  &\leq & C_e \sqrt{\delta} \mu^\alpha D^\alpha\label{e:a,al} \\
\|\partial_\tau a_k(\cdot,s,\tau)\|_\alpha + \|\partial_{\tau\tau} a_k (\cdot, s, \tau)\|_\alpha&\leq & C_v \sqrt{\delta} \mu^\alpha D^\alpha\label{e:a_t,al}\\
\|(\partial_\tau a_k+i(k\cdot v_{\ell})a_k)(\cdot,s,\tau)\|_\alpha &\leq & C_e \sqrt{\delta} \mu^{\alpha-1}D^\alpha
\label{e:a_trans,al}\\
\|\partial_\tau(\partial_\tau a_k+i(k\cdot v_{\ell})a_k)(\cdot,s,\tau)\|_\alpha &\leq & C_v \sqrt{\delta} \mu^{\alpha-1}D^\alpha\label{e:a_trans_t,al}
\end{eqnarray}
The following estimates hold for any $r\geq 0$:
\begin{eqnarray}
\|\partial_s a_k(\cdot,s,\tau)\|_{r}  &\leq& C_e \sqrt{\delta} \left(\mu^{r+1}D^{r+1}+\mu D\ell^{-r}\right)\label{e:a_s}\\
\|\partial_{s\tau} a_k (\cdot, s, \tau)\|_r &\leq& C_v \sqrt{\delta} \left(\mu^{r+1}D^{r+1}+\mu D\ell^{-r}\right)\label{e:a_st}\\
\|\partial_{ss} a_k(\cdot,s,\tau)\|_{r}&\leq& C_e \sqrt{\delta} \left(\mu^{r+2}D^{r+2} + \mu D \ell^{-1-r}\right)
\label{e:a_ss}\\
\|\partial_s(\partial_\tau a_k+i(k\cdot v_{\ell})a_k)(\cdot,s,\tau)\|_r&\leq & C_v \sqrt{\delta} \left(\mu^{r}D^{r+1}+ \mu D\ell^{-r}\right)\label{e:a_trans_s}
\end{eqnarray}

(ii) The matrix-function $W\otimes W$ can be written as
\begin{equation}\label{e:WoW}
(W\otimes W)(y,s,\tau,\xi)=R_\ell(y,s)+\sum_{1\leq |k| \leq 2\lambda_0}U_k(y,s,\tau)e^{ik\cdot \xi},
\end{equation}
where the coefficients $U_k\in C^{\infty}(\T^3\times\mathbb S^1\times\R;\S^{3\times 3})$ satisfy
\begin{equation}\label{e:Uk}
U_kk=\frac{1}{2}(\tr U_k)k\, .
\end{equation}
Moreover, we have the following estimates for any $r\geq 1$ and any $\alpha\in [0,1]$:
\begin{eqnarray}
\|U_k(\cdot,s,\tau)\|_r &\leq & C_e \delta\left(\mu^rD^r+\mu D\ell^{1-r}\right)\label{e:U,r}\\
\|\partial_\tau U_k(\cdot,s,\tau)\|_r &\leq& C_v \delta\left(\mu^rD^r+\mu D\ell^{1-r}\right)
\label{e:U_t,r}\\
\|U_k(\cdot,s,\tau)\|_\alpha &\leq & C_e \delta \mu^\alpha D^\alpha\label{e:U,al}\\
\|\partial_\tau U_k(\cdot,s,\tau)\|_\alpha&\leq & C_v \delta \mu^\alpha D^\alpha\label{e:U_t,al}
\end{eqnarray}
and the following estimate for any $r\geq 0$:
\begin{equation}
\|\partial_s U_k(\cdot,s,\tau)\|_{r}\leq C_e \delta \left(\mu^{r+1}D^{r+1}+\mu D\ell^{-r}\right)\, .\label{e:U_s}
\end{equation}
\end{proposition}

\begin{proof}
The arguments for \eqref{e:WoW} and \eqref{e:Uk} are analogous to those in the proof of \cite[Proposition 6.1]{DS3}. Moreover,
precisely as argued there, the estimates for the $U_k$ terms follow easily from the estimates for the $a_k$ coefficients, since each $U_k$ is the sum of finitely many terms of the form $a_{k'}a_{k''}$.
Here we focus, therefore, on the estimates \eqref{e:a,r}-\eqref{e:a_trans_s}.

\medskip

First of all observe that it suffices to prove the cases $r\in \mathbb N$, since the remaining ones can be obtained by interpolation. Recall now the formula for $a_k$: if $k\in \bigcup_j \Lambda_j$, then
\begin{equation}\label{e:a_k}
a_k = \sqrt{\rho_\ell(s)}\gamma_{k}^{(j)}\left(\frac{R_\ell(y,s)}{\rho_\ell(s)}\right)\phi_{k,\mu}^{(j)}\left(v_{\ell}(y,s),\tau\right)\, ,
\end{equation}
otherwise $a_k$ vanishes identically.

Observe that the functions $a_k$ depend on the variables $y$, $s$ and $\tau$. We introduce
the notation $\llbracket \cdot \rrbracket_m$ for the H\"older seminorms in $y$ {\em and} $s$
\[
\llbracket a_k (\cdot, \cdot, \tau)\rrbracket_m = \sum_{j+|\beta|=m} \left\|\partial^j_s D^\beta_y a_k\right\|_0\, 
\]
and the notation
$\hys a_k (\cdot, \cdot, \tau)\hys_m$ for the H\"older norm in $y$ {\em and} $s$:
\[
\hys a_k (\cdot,\cdot, \tau)\hys_m = \sum_{i=0}^m \llbracket a_k (\cdot, \cdot, \tau)\rrbracket_i \, .
\]
We next introduce the functions
\[
\Gamma (y,s) = \gamma_k^{(j)}\left(\frac{R_\ell(y,s)}{\rho_\ell(s)}\right) \qquad
\mbox{and}\qquad \Phi (y,s, \tau) =  \phi_{k,\mu}^{(j)}\left(v_{\ell}(y,s),\tau\right)
\]
and observe that
\[
a_k = \sqrt{\rho_\ell}\, \Gamma\, \Phi\, .
\]
Recall that $\|\rho_l\|_0\leq C_e \delta$ by \eqref{e:rho_above}. Therefore
the claimed estimate for $r=\alpha = 0$ follows trivially. Thus, we assume $r\in \mathbb N\setminus \{0\}$
and we focus on the estimates \eqref{e:a,r}-\eqref{e:a_trans_t,r} and \eqref{e:a_s}-\eqref{e:a_trans_s}.

\medskip

{\bf Proof of the estimates \eqref{e:a,r}, \eqref{e:a_s} and \eqref{e:a_ss}.} Recalling \eqref{e:Holderproduct}, we estimate
\begin{align}
\hys a_k \hys_r&\leq C_h\|\sqrt{\rho_\ell}\|_0 \|\Gamma\|_0\llbracket\Phi\rrbracket_r+C_h \|\sqrt{\rho_\ell}\|_0 \|\Phi\|_0 \llbracket\Gamma\rrbracket_r + C_h \|\Phi\|_0\|\Gamma\|_0 \llbracket \sqrt{\rho_\ell}\rrbracket_r\, \nonumber\\
& \leq C_e \left( \sqrt{\delta} \left(\llbracket \Phi\rrbracket_r + \llbracket \Gamma \rrbracket_r\right) + \llbracket \sqrt{\rho_\ell}\rrbracket_r\right)\, .\label{e:triple_prod}
\end{align}
Next, by \eqref{e:conv1_stand}, for any $j\geq 1$ we have $[v_\ell]_j \leq C_h D \ell^{1-j}$ for
every $j\geq 1$.
Applying \eqref{e:chain1} in Proposition \ref{p:chain} and Proposition \ref{p:phi} we conclude
\begin{align}
\llbracket\Phi\rrbracket_r &\leq C_h \sum_{i=1}^r  \left[\phi^{(j)}_{k, \mu}\right]_i [v_\ell]_1^{(i-1) \frac{r}{r-1}} [v_\ell]_r^{\frac{r-i}{r-1}} \leq C_h \sum_{i=1}^r  \left[\phi^{(j)}_{k, \mu}\right]_i D^i \ell^{i-r}\nonumber\\ 
&\stackrel{\eqref{e:phi}}{\leq} C_h \sum_{i=1}^r C_h \mu^i D^i \ell^{-r+i}\leq C_h \left(\mu^r D^r + \mu D \ell^{1-r}\right)\, .\label{e:est_Phi}
\end{align}
Applying \eqref{e:chain0} of Proposition \ref{p:chain} we also conclude
\begin{align}
\llbracket\Gamma\rrbracket_r 
&\leq C_h \sum_{i=1}^r \left[\gamma_k^{(j)}\right]_i \left\|\frac{R_\ell}{\rho_\ell}\right\|_0^{i-1}
\left\llbracket\frac{R_\ell}{\rho_\ell}\right\rrbracket_r\label{e:Gamma_1}
\end{align}
Now, by \eqref{e:r_0/2} we have
\[
\left\|\frac{R_\ell}{\rho_\ell}\right\|_0 \leq \frac{r_0}{2}+1\, .
\]
Moreover $[\gamma^{(j)}_k]_r \leq C_h$:
indeed recall that, because of our choice of $\eta$ in Section \ref{s:constants}, the range of $\frac{R_\ell}{\rho_\ell}$ is contained
in $B_{\frac{r_0}{2}} ({\rm Id})$, whereas the $\gamma^{(j)}_k$ are defined on the open ball $B_{r_0} ({\rm Id})$;
since the $\gamma^{(j)}_k$ are smooth and finitely many, obviously we can bound their norms
uniformly on the range of the function $\frac{R_\ell}{\rho_\ell}$. 

Using these estimates in \eqref{e:Gamma_1}
we thus get
\begin{equation}\label{e:Gamma_2}
\llbracket \Gamma\rrbracket_r \leq C_h \left\llbracket\frac{R_\ell}{\rho_\ell}\right\rrbracket_r
\stackrel{\eqref{e:Holderproduct}}{\leq}
\|\rho_\ell^{-1}\|_0 \llbracket R_\ell \rrbracket_r
+ \|R_\ell\|_0 \llbracket\rho_\ell^{-1}\rrbracket_r\, .
\end{equation}
Recall next that, by \eqref{e:rho_below}, $\rho_\ell (s)\geq C_e \delta$ for every $s$.
Moreover, by \eqref{e:def_rho}, for $r\geq 1$ we have 
\[
\partial^r_s \rho_\ell (s) = \frac{1}{3(2\pi)^3}\left((1-\bar{\delta})\partial_s^re(s)-\sum_{j=0}^r \binom{r}{j} 
\int_{\T^3} \left(\partial_s^j v_\ell \cdot \partial_s^{r-j} v_\ell\right) (x,s)\, dx\right)\, .
\]
Thus, we conclude
\begin{align}
[\rho_\ell]_r &\leq C_e+C \|v_\ell\|_{C^0_t L^2_x} [v_\ell]_r + C_h \sum_{j=1}^{r-1} [v_\ell]_j [v_\ell]_{r-j}\nonumber\\
&\leq C_e+C_e [v_\ell]_r + C_h \sum_{j=1}^{r-1} [v_\ell]_j [v_\ell]_{r-j}\nonumber\\ 
&\stackrel{\eqref{e:conv1_stand}}{\leq} C_e D \ell^{1-r} + C_h D^2 \ell^{r-2}
\stackrel{\eqref{e:range}}{\leq} C_e D \ell^{1-r}\, . \label{e:rho_r}
\end{align}

Set $\Psi (\zeta) = \zeta^{-1}$. On the domain $[\delta, \infty[$, we have the estimate
$[\Psi]_i \leq C_h \delta^{-i-1}$. Therefore, applying again \eqref{e:chain0} we conclude
\begin{align}
\llbracket\rho_\ell^{-1}\rrbracket_r &\leq C_h \sum_{i=1}^r \delta^{-i-1} \|\rho_\ell\|_0^{i-1} [\rho_\ell]_r
\leq C_h \delta^{-2} [\rho_\ell]_r \leq C_e \delta^{-2} D \ell^{r-1}\, .\label{e:rho^-1_r}
\end{align}
It follows from \eqref{e:Gamma_2}, \eqref{e:rho^-1_r} and \eqref{e:conv1_stand}  that
\begin{equation}\label{e:est_Gamma}
\llbracket\Gamma\rrbracket_r \leq C_e \delta^{-1} D \ell^{r-1}\, .
\end{equation}

\medskip

Next, set $\Psi (\zeta)= \zeta^{\frac{1}{2}}$. In this case, on the domain $[\delta, C_e \delta[$
we have the estimates $[\Psi]_i \leq C_e \delta^{\frac{1}{2}-i}$. Thus, by \eqref{e:chain0}
and \eqref{e:rho_r}:
\begin{equation}\label{e:est_sqrt_rho}
\llbracket \sqrt{\rho_\ell}\rrbracket_r \leq C_h \sum_{i=1}^r C_e \delta^{\frac{1}{2}-i} \|\rho_\ell\|_0^{i-1}
[\rho_\ell]_r \leq C_e \delta^{-\frac{1}{2}} D \ell^{1-r}\, .
\end{equation}
Inserting \eqref{e:est_Phi}, \eqref{e:est_Gamma} and \eqref{e:est_sqrt_rho}
into \eqref{e:triple_prod} we conclude
\[
|||a_k|||_r \leq C_e \delta^{-\frac{1}{2}} D \ell^{1-r} +C_e \delta^{\frac{1}{2}} \mu^r D^r +
C_e \delta^{\frac{1}{2}} \mu D \ell^{1-r}\, .
\]
Recall, however, that $\mu\geq \delta^{-1}$ and hence
\[
|||a_k|||_r \leq C_e \sqrt{\delta} \left( \mu^r D^r + \mu D \ell^{1-r}\right)\, .
\]
From this we derive the claimed estimates for $\|a_k\|_r$ for any $r\geq 1$ and for $\|\partial_s a_k\|_r$ and
$\|\partial_{ss} a_k\|_r$ for any $r\geq 0$.

\medskip

{\bf Proof of the estimates \eqref{e:a_t,r} and \eqref{e:a_st}.} 
Differentiating in $\tau$ we obtain the identities 
\begin{align*}
\partial_\tau a_k (\cdot, \cdot, \tau) &= \sqrt{\rho_\ell}\; \Gamma\, \partial_\tau \phi^{(j)}_{k, \mu} (v_\ell, \tau)\, \\
\partial_{\tau\tau} a_k (\cdot, \cdot, \tau) &= \sqrt{\rho_\ell}\; \Gamma\, \partial_{\tau\tau} 
\phi^{(j)}_{k, \mu} (v_\ell, \tau)\, .
\end{align*}
Thus, arguing precisely as above, we achieve the desired estimates for the quantities $\|\partial_\tau a_k\|_r$, $\|\partial_{\tau s} a_k\|_r$ and $\|\partial_{\tau\tau} a_k\|_r$. However, note that we use
the estimate \eqref{e:phi} with $R:= \|v\|_0$ and for $[\partial_t \phi^{(j)}_{k,\mu}]_{m,R}$ and
$[\partial_{\tau\tau} \phi^{(j)}_{k,\mu}]_{m,R}$. It turns out, therefore, that the constants in the estimates 
\eqref{e:a_t,r} and \eqref{e:a_st} depend also on $\|v\|_0$.

\medskip

{\bf Proof of the estimates \eqref{e:a_trans,r}, \eqref{e:a_trans_t,r} and \eqref{e:a_trans_s}.}
Finally, we introduce the function
\[
{\chi}^{(j)}_{k, \mu} (v, \tau) := \partial_\tau \phi^{(j)}_{k, \mu} + i(k\cdot v) \phi^{(j)}_{k, \mu}\, 
\]
and
${\chi} (y,s, \tau) = \chi^{(j)}_{k, \mu} (v_\ell (y,s), \tau)$.
Then
\[
\partial_\tau a_k + i (k\cdot v_\ell) a_k = \sqrt{\rho_\ell} \chi \Gamma\, .
\]
Applying the same computations as above and using the estimates in Proposition \ref{p:phi} we
achieve the desired estimates for $\|\partial_\tau a_k + i (k\cdot v_\ell) a_k\|_r$ and
$\|\partial_s (\partial_\tau a_k + i (k\cdot v_\ell) a_k)\|_r$. Finally,
\[
\partial_\tau (\partial_\tau a_k + i (k\cdot v_\ell) a_k )
= \sqrt{\rho_\ell}\; \Gamma \left[ \partial_\tau \chi^{(j)}_{k, \mu}\right] (v_\ell, \tau)
\]
and hence the arguments above carry over to estimate also the quantity $\|\partial_\tau (\partial_\tau a_k + i (k\cdot v_\ell) a_k)\|_r$.
\end{proof}

\section{Estimates on $w_o$, $w_c$ and $v_1$}\label{s:est3}

\begin{proposition}\label{p:velocity_est}
Under assumption \eqref{e:range}, the following estimates hold for any $r\geq 0$
\begin{align}
\|w_o\|_r &\leq C_e \sqrt{\delta} \lambda^r\,, \label{e:est_w_o}\\
\|\partial_t w_o\|_r&\leq C_v \sqrt{\delta}\lambda^{r+1}\,  \label{e:est_w_o_t}
\end{align}
and the following for any $r>0$ {\em which is not integer}:
\begin{align}
\|w_c\|_r &\leq C_{e,s} \sqrt{\delta} D \mu\, \lambda^{r-1}\label{e:est_w_c}\\
\|\partial_t w_c\|_r &\leq C_{v,s}\sqrt{\delta} D \mu \lambda^r\label{e:est_w_c_t}\, . 
\end{align}
In particular 
\begin{align}
\|w\|_0&\leq C_e\sqrt{\delta}\,,\label{e:w_C0}\\
\|w\|_{C^1}&\leq C_{v}\sqrt{\delta}\lambda\,.\label{e:w_C1}
\end{align}
\end{proposition}

\begin{proof} First of all observe that it suffices to prove \eqref{e:est_w_o} when $r=m\in\N$,
since the remaining inequalities can be obtained by interpolation. By writing
\begin{align*}
w_o (x,t) &= \sum_{|k|=\lambda_0} a_k (x,t,\lambda t) B_k e^{i\lambda k\cdot x}
=: \sum_{|k|=\lambda_0} a_k (x, t, \lambda t) \Omega _k (\lambda x),\\
\partial_t w_o (x,t) &= \lambda \sum_{|k|= \lambda_0} \partial_\tau a_k (x,t, \lambda t) \Omega_k (\lambda x)
+ \sum_{|k|= \lambda_0} \partial_s a_k (x,t, \lambda t) \Omega_k (\lambda x),
\end{align*}
from \eqref{e:Holderproduct} we obtain
\begin{align*}
\|w_o\|_m &\leq C_h \sum_{|k|= \lambda _0} \left( \|\Omega_k\|_0 [a_k]_m + \lambda^m \|a_k\|_0 [\Omega_k]_m\right),\\
\|\partial_t w_o\|_m &\leq C_h \lambda \sum_{|k|= \lambda _0} \left( \|\Omega_k\|_0 [\partial _\tau a_k]_m + \lambda^m \|\partial_\tau a_k\|_0 [\Omega_k]_m\right)\\
&\quad + C_h \sum_{|k|= \lambda _0} \left( \|\Omega_k\|_0 [\partial_s a_k]_m + \lambda^m \|\partial_s a_k\|_0 [\Omega_k]_m\right).
\end{align*}
When $m=0$, we then use \eqref{e:a,al} to conclude \eqref{e:est_w_o} and \eqref{e:a_t,al}
and \eqref{e:a_s} to conclude \eqref{e:est_w_o_t}. For $m\geq 1$ we use, respectively, \eqref{e:a,r}
and the estimates \eqref{e:a_t,r} and \eqref{e:a_s} to get:
\begin{align*}
\|w_o\|_m&\leq C_e \sqrt{\delta} \left( \mu^m D^m + \mu D \ell^{1-m}
+ \lambda^m\right)\\
\|\partial_t w_o\|_m &\leq C_v  \sqrt{\delta} \Bigl( \lambda\mu^m D^m +\lambda \mu D \ell^{1-m}+ \lambda^{m+1}\\
&\qquad +\mu^{m+1}D^{m+1}+\mu D\ell^{-m}+\lambda^m\mu D\Bigr)
\end{align*}
However, recall from \eqref{e:range} that $\lambda \geq (D\mu)^{1+\omega} \geq D\mu$ and
$\lambda \geq \ell^{-1}$. Thus \eqref{e:est_w_o} and \eqref{e:est_w_o_t} follow easily.

As for the estimates on $w_c$ we argue as in \cite[Lemma 6.2]{DS3} and 
start with the observation that, since $k\cdot B_k=0$, 
\begin{equation*}
\begin{split}
w_o(x,t)=&\frac{1}{\lambda}\nabla\times\left(\sum_{|k|=\lambda_0}-ia_k(x,t,\lambda t)\frac{k\times B_k}{|k|^2}e^{i\lambda x\cdot k}\right)+\\
&+\frac{1}{\lambda}\sum_{|k|=\lambda_0}i\nabla a_k(x,t,\lambda t)\times \frac{k\times B_k}{|k|^2}e^{i\lambda x\cdot k}.
\end{split}
\end{equation*}
Hence
\begin{equation}\label{e:w_cu_c}
w_c(x,t)=\frac{1}{\lambda}\Q u_c(x,t),
\end{equation}
where 
\begin{equation}\label{e:wo_trick}
u_c(x,t)=\sum_{|k|=\lambda_0}i\nabla a_k(x,t,\lambda t)\times \frac{k\times B_k}{|k|^2}e^{i\lambda x\cdot k}.
\end{equation}
The Schauder estimate \eqref{e:Schauder_Q} gives then 
\begin{equation}\label{e:Schauder_again}
\|w_c\|_{m+\alpha} \leq \frac{C_s}{\lambda} \|u_c\|_{m+\alpha}
\end{equation}
for any $m\in\N$ and $\alpha\in(0,1)$.
We next wish to estimate $\|u_c\|_r$. For integer $m$ we can argue as for the estimate
of $\|w_o\|$ to get
\begin{align*}
\|u_c\|_m &\leq C_e \left([a_k]_1 \lambda^m + [a_k]_{m+1}\right) \leq C_e \sqrt{\delta} \left( \mu D \lambda^m
+ \mu D \ell^{-m}\right)\\ 
&\leq C_e \sqrt{\delta}\mu D \lambda^m\, .
\end{align*}
Hence, by interpolation, we reach the estimate
$\|u_c\|_{m+\alpha} \leq C_e \sqrt{\delta}\mu D \lambda^{m+\alpha}$ for any $m,\alpha$. Combining this with
\eqref{e:Schauder_again}, for $r>0$ which is not an integer we conclude 
$\|w_c\|_{r} \leq C_{e,s} \sqrt{\delta} \mu D \lambda^{r-1}$. 

Similarly, for $\partial_t w_c$ we have
\[
\partial_t w_c = \frac{1}{\lambda} \mathcal{Q} \partial_t u_c\, .
\]
Differentiating \eqref{e:wo_trick} we achieve
\begin{align*}
\partial_t u_c(x,t)& = \lambda \sum_{|k|=\lambda_0}i\nabla \partial_\tau a_k(x,t,\lambda t)\times \frac{k\times B_k}{|k|^2}e^{i\lambda x\cdot k}\\
&\quad + \sum_{|k|=\lambda_0}i\nabla \partial_s a_k(x,t,\lambda t)\times \frac{k\times B_k}{|k|^2}e^{i\lambda x\cdot k}\, .
\end{align*}
Using Proposition \ref{p:W} and \eqref{e:range} we deduce, analogously to above, $\|\partial_t u_c\|_r \leq C_v \sqrt{\delta} \mu D \lambda^{r+1}$. Using \eqref{e:Schauder_again} once more we arrive at \eqref{e:est_w_c}.

\medskip

To obtain \eqref{e:w_C0} and \eqref{e:w_C1}, recall that $w=w_o+w_c$. 
For any $\alpha>0$ we therefore have
\begin{equation}\label{e:al_to_choose}
\|w\|_0\leq \|w_o\|_0+\|w_c\|_\alpha\leq C_e\sqrt{\delta}+C_{e,s}\sqrt{\delta}D\mu\lambda^{\alpha-1}.
\end{equation}
We now use \eqref{e:al_to_choose} with $\alpha =\frac{\omega}{1+\omega}$: since by \eqref{e:range} we have $\lambda^{1-\alpha}= \lambda^{\frac{1}{1+\omega}}\geq D\mu$,
\eqref{e:w_C0} follows. In the same way
\begin{equation*}
\begin{split}
\|w\|_{C^1}&\leq \|w_o\|_{1}+\|\partial_tw_o\|_{0}+\|w_c\|_{1+\alpha}+\|\partial_tw_c\|_{\alpha}\\
&\leq C_v\sqrt{\delta}\lambda+C_{v,s}\sqrt{\delta}D\mu\lambda^{\alpha}\, .
\end{split}
\end{equation*}
Again choosing $\alpha = \frac{\omega}{1+\omega}$ and arguing as above we conclude \eqref{e:w_C1}.
\end{proof}

\section{Estimate on the energy}\label{s:est4}

\begin{proposition}\label{p:energy}
For any $\alpha\in (0,\frac{\omega}{1+\omega})$ there is a constant $C_{v,s}$, depending only on $\alpha$, $e$
and $\|v\|_0$, such that,
if the parameters satisfy \eqref{e:range}, then
\begin{equation}\label{e:energy_error}
\left| e(t)(1-\bar{\delta}) - \int |v_1|^2 (x,t)\, dx\right| \leq C_e D \ell + C_{v,s} 
\sqrt{\delta} \mu D\lambda^{\alpha -1} \qquad \forall t\, .
\end{equation}
\end{proposition}
\begin{proof}
We write
\begin{equation}\label{e:quadratic_exp}
|v_1|^2 = |v|^2 + |w_o|^2 + |w_c|^2 + 2 w_o\cdot v + 2 w_o\cdot w_c + 2 w_c\cdot v\, .
\end{equation}
Since
\[
\left|\int w_c\cdot v\right|\leq \|w_c\|_0 \|v (\cdot, t)\|_{L^2} \leq \sqrt{e (t)} \|w_c\|_0\,,
\]
integrating the identity \eqref{e:quadratic_exp} we then reach the inequality
\[
\left| \int (|v_1|^2 - |w_o|^2 - |v|^2) \, dx \right| \leq C_e \|w_c\|_0 (1 + \|w_c\|_0 + \|w_o\|_0) + 2 \left| \int w_o\cdot v \right|\, .
\]
By Proposition \ref{p:velocity_est} we then have
\begin{align*}
\left| \int (|v_1|^2 - |w_o|^2 - |v|^2) \, dx \right| 
&\leq C_{e,s} \sqrt{\delta} D \mu \lambda^{\alpha-1} \left( 1 + C_e \sqrt{\delta} D \mu \lambda^{\alpha-1} + C_e \sqrt{\delta}\right)\\
&\quad + 2 \left| \int w_o\cdot v \right|\, 
\end{align*}
and hence, recalling that $\lambda \geq (D\mu)^{1+\omega}$ we reach
\[
\left| \int (|v_1|^2 - |w_o|^2 - |v|^2) \, dx \right| 
\leq C_{e,s} \sqrt{\delta} D \mu \lambda^{\alpha-1}
+ 2 \left| \int w_o\cdot v \right|
\]
Applying Proposition \ref{p:schauder}(i) and Proposition \ref{p:W} we obtain
\[
\left| \int w_o\cdot v \right| \leq C_e \sum_{k=|\lambda_0|} \frac{[v a_k]_1}{\lambda} \leq 
C_e \|v\|_0 \sqrt{\delta} D\mu \lambda^{-1} + C_e D \sqrt{\delta}\lambda^{-1}\, ,
\]
and hence
\begin{equation}\label{e:en_part_1}
\left| \int (|v_1|^2 - |w_o|^2-|v|^2)\right| \leq C_{v,s} \sqrt{\delta} D \mu \lambda^{\alpha-1}\, .
\end{equation}

\medskip

Next, taking the trace of identity \eqref{e:WoW} in Proposition \ref{p:W} we have
\[
|W(y,s,\tau,\xi)|^2=\tr R_\ell (y,s)+\sum_{1\leq |k| \leq 2\lambda_0}c_k(y,s,\tau)e^{ik\cdot \xi}
\]
for the coefficients $c_k = \tr U_k$. Recall that 
\[
\int_{\T^3} \tr R_\ell(x,t)\,dx = 3(2\pi)^3\rho_\ell(t)=e(t) (1-\bar{\delta}) -\int_{\T^3} |v_{\ell}|^2\,dx.
\]
Moreover, by Proposition \ref{p:schauder}(i) with $m=1$ we have
\begin{align}
\left| \int (|w_o|^2 (x,t) - {\rm tr}\, R_\ell (x,t)\, dx \right|&\leq
\sum_{1\leq |k| \leq 2\lambda_0} \left|\int c_k (x,t, \lambda t) e^{ik\cdot \lambda x}\,dx\right|\nonumber\\
&\leq C \lambda^{-1} \sum_{1\leq |k|\leq 2\lambda_0} [c_k]_1 \stackrel{\eqref{e:U,r}}{\leq} C_e \delta D 
\mu \lambda^{-1}\, .
\end{align}
Thus we conclude
\begin{equation}\label{e:en_part_2}
\left|\int \left(|w_o|^2 + |v_{\ell}|^2\right)\,dx - e(t) (1-\bar{\delta}) \right| \leq C_e \delta D \mu \lambda^{-1}\, .
\end{equation}
Finally, recall from \eqref{e:L2norm_est} that
\begin{equation}\label{e:en_part_3}
\left|\int (|v|^2 - |v_{\ell}|^2)\right|\leq C_e D \ell\, .
\end{equation}
Putting \eqref{e:en_part_1}, \eqref{e:en_part_2} and \eqref{e:en_part_3} together, we achieve
\eqref{e:energy_error}.
\end{proof}

\section{Estimates on the Reynolds stress}\label{s:est5}

\begin{proposition}\label{p:Reynolds}
For every $\alpha\in (0, \frac{\omega}{1+\omega})$, there is a constant $C_{v,s}$, depending only on $\alpha$, $\omega$, $e$
and $\|v\|_0$, such that, if
the conditions \eqref{e:range} are satisfied, then the following estimates hold:
\begin{eqnarray}
\|\mathring{R}_1\|_0&\leq& C_{v,s} \left(D\ell + \sqrt{\delta} D \mu \lambda^{2\alpha-1} + \sqrt{\delta} \mu^{-1}\lambda^\alpha \right)\label{e:est_Reyn}\\
\|\mathring{R}_1\|_{C^1} &\leq& C_{v,s} \lambda \left(\sqrt{\delta} D \ell + \sqrt{\delta} D \mu \lambda^{2\alpha-1} + \sqrt{\delta} \mu^{-1}\lambda^{\alpha} \right) \label{e:est_DReyn}\, .
\end{eqnarray}
\end{proposition}
\begin{proof} We split the Reynolds stress into seven parts:
\[
\mathring{R}_1 = \mathring{R}^1_1 + \mathring{R}^2_1 + \mathring{R}^3_1 + \mathring{R}^4_1 + \mathring{R}^5_1
+ \mathring{R}^6_1 + \mathring{R}^7_1
\]
where
\begin{eqnarray*}
\mathring{R}^1_1 &=& \mathring{R}_\ell - \mathring{R}\\
\mathring{R}^2_1 &=& [w\otimes (v - v_{\ell}) + (v-v_{\ell})\otimes w
 - \textstyle{\frac{2 \langle (v-v_{\ell}), w\rangle}{3}} {\rm Id}]\\
\mathring{R}^3_1 &=& \RR [{\rm div} (w_o\otimes w_o + \mathring{R}_{\ell} - \textstyle{\frac{|w_o|^2}{2}}{\rm Id})]\\
\mathring{R}^4_1 &=& \RR \partial_t w_c\\
\mathring{R}^5_1 &=& \RR {\rm div} ((v_\ell +w) \otimes w_c + w_c\otimes (v_\ell +w) - w_c\otimes w_c)\\
\mathring{R}^6_1 &=& \RR {\rm div} (w_o \otimes v_{\ell})\\
\mathring{R}^7_1 &=& \RR [ \partial_t w_o + {\rm div} (v_{\ell} \otimes w_o)] = \RR [\partial_t w_o + v_{\ell} \cdot \nabla w_o]\, .
\end{eqnarray*} 
In what follows we will estimate each term separately in the order given above.

\medskip

{\bf Step 1.} Recalling \eqref{e:conv0_stand}:
\begin{eqnarray}
\|\mathring{R}^1_1\|_0 &\leq& C D\ell \label{e:one}\\
\|\mathring{R}^1_1\|_{C^1} &\leq & 2D \label{e:one_D}\, .
\end{eqnarray}

\medskip

{\bf Step 2.} Again by \eqref{e:conv0_stand} and \eqref{e:conv1_stand}:
\begin{eqnarray*}
\|v-v_{\ell}\|_0 &\leq& C D \ell\\
\|v-v_{\ell}\|_{C^1} &\leq& 2D\, .
\end{eqnarray*}
Moreover, Proposition \ref{p:velocity_est} gives
\begin{align*}
\|w\|_0 &\leq C_{e} \sqrt{\delta}\\
\|w\|_{C^1}&\leq C_v \sqrt{\delta}{\lambda}\, .
\end{align*}
Using this and \eqref{e:Holderproduct} we conclude
\begin{eqnarray}
\|\mathring{R}^2_1\|_0 &\leq& C_ e \sqrt{\delta} D \ell \label{e:two}\\
\|\mathring{R}^2_1\|_{C^1} &\leq& C_e \sqrt{\delta} D + C_v \sqrt{\delta} \lambda D \ell
\leq C_v \sqrt{\delta} \lambda D \ell \label{e:two_D}\, .
\end{eqnarray}

\medskip

{\bf Step 3.} We next argue as in the proof of \cite[Lemma 7.2]{DS3}. Recall the formula \eqref{e:WoW} from Proposition \ref{p:W}. 
Since $\rho_\ell$ is a function of $t$ only, we can write $\mathring{R}^3_1$ as
\begin{align}
\div(w_o\otimes w_o-&\tfrac{1}{2}(|w_o|^2-\rho_\ell)\textrm{Id}+\mathring{R}_\ell)\nonumber\\
&=\div\left(w_o\otimes w_o-R_\ell-\tfrac{1}{2}(|w_o|^2-\tr R_\ell)\textrm{Id}\right)\nonumber\\
&=\div\left[\sum_{1\leq |k| \leq 2\lambda_0}(U_k-\tfrac{1}{2}(\tr U_k)\textrm{Id})(x,t,\lambda t)e^{i\lambda k\cdot x}\right]\nonumber\\
&\stackrel{\eqref{e:Uk}}{=}\sum_{1\leq |k| \leq 2\lambda_0}\div_y[U_k-\tfrac{1}{2}(\tr U_k)\textrm{Id}] (x, t, \lambda t) e^{i\lambda k\cdot x}\, .\label{e:third_dwarf}
\end{align}
We can therefore apply Proposition \ref{p:schauder} with 
\begin{equation}\label{e:choice_of_m}
m = \left\lfloor \frac{1+\omega}{\omega}\right\rfloor +1\, 
\end{equation}
and $\alpha\in (0,\frac{\omega}{1+\omega})$. Combining the corresponding estimates with Proposition \ref{p:W} we get
\begin{align}
\|\mathring{R}^3_1\|_0 &\leq C_s(m,\alpha) \sum_{1\leq |k| \leq 2\lambda_0} \left(\lambda^{\alpha-1} [U_k]_1
+ \lambda^{\alpha-m} [U_k]_{m+1} + \lambda^{-m} [U_k]_{m+1+\alpha}\right) \nonumber\\
&\leq C_s(m,\alpha)C_e \Bigl(\lambda^{\alpha-1} \delta \mu D + \lambda^{\alpha-m} \delta \left( \mu^{m+1} D^{m+1} + \mu D \ell^{-m}\right)\nonumber\\ 
&\qquad +  \lambda^{-m} \delta \left(\mu^{m+1+\alpha} D^{m+1+\alpha} + \mu D \ell^{-m-\alpha}\right)\Bigr)\nonumber\\
&\overset{\eqref{e:range}}{\leq} C_{e,s} \delta \mu D \lambda^{\alpha-1}\label{e:three}\, .
\end{align}

Observe that in the last inequality we have used \eqref{e:range}: indeed, since $m\geq \frac{1+\omega}{\omega}$ by \eqref{e:choice_of_m}, we get
\begin{equation}\label{e:reason_for_m}
\lambda\geq \max\left\{\ell^{-(1+\omega)},\,(\mu D)^{1+\omega}\right\}\geq \max\left\{\ell^{-\frac{m}{m-1}},\,(\mu D)^{\frac{m}{m-1}}\right\}\,.
\end{equation} 
Next, differentiating \eqref{e:third_dwarf} in space and using the same argument:
\begin{align*}
\|\mathring{R}^3_1\|_1 &\leq C_e \lambda \|\mathring{R}^3_1\|_0\nonumber\\ 
&\quad + C_s \sum_{1\leq |k| \leq 2\lambda_0} \left(\lambda^{\alpha-1} [U_k]_2
+ \lambda^{\alpha-m} [U_k]_{m+2} + \lambda^{-m} [U_k]_{m+2+\alpha}\right) \nonumber\\
&\leq C_{e,s} \delta \mu D \lambda^\alpha.
\end{align*}
Finally, differentiating \eqref{e:third_dwarf} in time:
\begin{align*}
\partial_t \div(w_o\otimes w_o-&\tfrac{1}{2}(|w_o|^2-\rho_\ell)\textrm{Id}+\mathring{R}_\ell)\\
&=\sum_{1\leq |k| \leq 2\lambda_0}\div_y[\partial_s U_k-\tfrac{1}{2}(\tr \partial_s U_k)\textrm{Id}] (x, t, \lambda t) e^{i\lambda k\cdot x}\nonumber\\
&\quad +
\lambda \sum_{1\leq |k| \leq 2\lambda_0}\div_y[\partial_\tau U_k-\tfrac{1}{2}(\tr \partial_\tau U_k)\textrm{Id}] (x, t, \lambda t) e^{i\lambda k\cdot x}\, .
\end{align*}
Thus, applying the same argument as above,
\begin{align*}
\|\partial_t \mathring{R}^3_1\|_0 &\leq C_s\sum_{1\leq |k| \leq 2\lambda_0}
\left(\lambda^{\alpha-1} [\partial_s U_k]_1 + \lambda^{\alpha-m} [\partial_s U_k]_{m+1}
+ \lambda^{-m} [\partial_s U_k]_{m+1+\alpha}\right)\\
&\qquad + C_s \lambda\sum_{1\leq |k| \leq 2\lambda_0}
\big(\lambda^{\alpha-1} [\partial_\tau U_k]_1 + \lambda^{\alpha-m} [\partial_\tau U_k]_{m+1}\\
&\qquad\qquad\qquad\qquad\qquad+ \lambda^{-m} [\partial_\tau U_k]_{m+1+\alpha}\big)\\
&\leq C_{v,s} (\mu D + \ell^{-1}+ \lambda) \delta \mu D \lambda^{\alpha-1} \\
&\leq C_{v,s} \delta \mu D \lambda^\alpha\, .
\end{align*}
Finally, putting these last two estimates together:
\begin{equation}\label{e:three_D}
\|\mathring{R}^3_1\|_{C^1} \leq \|\mathring{R}^3_1\|_1 + \|\partial_t \mathring{R}^3_1\|_0
\leq C_{v,s} \delta \mu D \lambda^\alpha\, .
\end{equation}

\medskip

{\bf Step 4.} In this case we argue as in \cite[Lemma 7.3]{DS3}. Differentiate in $t$ the identity \eqref{e:w_cu_c} to
get 
\[
\partial_t w_c=\tfrac{1}{\lambda}\Q \partial_t u_c\, ,
\]
where
\begin{equation*}
\begin{split}
\partial_tu_c(x,t)=&\lambda \sum_{|k|=\lambda_0}i(\nabla \partial_\tau a_k)(x,t,\lambda t)\times \frac{k\times B_k}{|k|^2}e^{i\lambda x\cdot k}+\\
&\,+\sum_{|k|=\lambda_0}i(\nabla \partial_s a_k)(x,t,\lambda t)\times \frac{k\times B_k}{|k|^2}e^{i\lambda x\cdot k}\, .
\end{split}
\end{equation*}
Choose again $m$ as in \eqref{e:choice_of_m} and apply the Propositions \ref{p:schauder} and
\ref{p:W} to get
\begin{align}
\|\mathring{R}^4_1\|_0&\leq C_s \sum_{|k| =\lambda_0}
\left(\lambda^{\alpha-1} [\partial_\tau a_k]_1 + \lambda^{\alpha-m} [\partial_\tau a_k]_{m+1}
+ \lambda^{-m} [\partial_\tau a_k]_{m+1+\alpha}\right)\nonumber\\
& + \frac{C_s}{\lambda} \sum_{|k| =\lambda_0}
\left(\lambda^{\alpha-1} [\partial_s a_k]_1 + \lambda^{\alpha-m} [\partial_s a_k]_{m+1}
+ \lambda^{-m} [\partial_s a_k]_{m+1+\alpha}\right)\nonumber\\
\leq &C_v (\lambda^{-1} \mu D + \lambda^{-1} \ell^{-1}+ 1) \sqrt{\delta} \mu D \lambda^{\alpha-1} 
\leq C_v \sqrt{\delta} \mu D \lambda^{\alpha-1}\, ,\label{e:four}
\end{align}
where in the last inequality we have again used \eqref{e:reason_for_m}. 
Following the same strategy as in Step 3:
\begin{align}
\|\mathring{R}^4_1\|_1&\leq C_e \lambda \|\mathring{R}^4_1\|_0\nonumber\\
&\quad + C_s \sum_{|k| =\lambda_0}
\left(\lambda^{\alpha-1} [\partial_\tau a_k]_2 + \lambda^{\alpha-m} [\partial_\tau a_k]_{m+2}
+ \lambda^{-m} [\partial_\tau a_k]_{m+2+\alpha}\right)\nonumber\\
&\quad + \frac{C_s}{\lambda} \sum_{|k| = \lambda_0}
\big(\lambda^{\alpha-1} [\partial_s a_k]_2 + \lambda^{\alpha-m} [\partial_s a_k]_{m+2}+ \lambda^{-m} [\partial_s a_k]_{m+2+\alpha}\big)\nonumber\\
&\leq C_{v,s} \sqrt{\delta} \mu D \lambda^\alpha\label{e:dwarf_four}\, .
\end{align}
Differentiating in time
\begin{align}
&\|\partial_t \mathring{R}^4_1\|_0 \nonumber\\ 
&\,\leq C_s \lambda \sum_{|k| = \lambda_0}
\left(\lambda^{\alpha-1} [\partial_{\tau\tau} a_k]_1 + \lambda^{\alpha-m} [\partial_{\tau\tau} a_k]_{m+1}
+ \lambda^{-m} [\partial_{\tau\tau} a_k]_{m+1+\alpha}\right)\nonumber\\
&\quad+ C_s \sum_{|k| = \lambda_0}
\left(\lambda^{\alpha-1} [\partial_{\tau s} a_k]_1 + \lambda^{\alpha-m} [\partial_{\tau s} a_k]_{m+1}
+ \lambda^{-m} [\partial_{\tau s} a_k]_{m+1+\alpha}\right)\nonumber\\
&\quad+\frac{C_s}{\lambda} \sum_{|k| = \lambda_0}
\left(\lambda^{\alpha-1} [\partial_{ss} a_k]_1 + \lambda^{\alpha-m} [\partial_{ss} a_k]_{m+1}
+ \lambda^{-m} [\partial_{ss} a_k]_{m+1+\alpha}\right) \nonumber\\
&\leq  C_{v,s} \sqrt{\delta} \mu D \lambda^\alpha\label{e:dwarf_four_2}\, .
\end{align}
Putting \eqref{e:dwarf_four} and \eqref{e:dwarf_four_2} together we obtain
\begin{equation}\label{e:four_D}
\|\mathring{R}^4_1\|_{C^1} \leq C_{v,s} \sqrt{\delta}\mu D \lambda^\alpha\, .
\end{equation}

\medskip

{\bf Step 5.} In this step we argue as in \cite[Lemma 7.4]{DS3}. We first estimate
\begin{equation*}
\begin{split}
\|(v_\ell + w)\otimes w_c+&w_c\otimes (v_\ell +w) -w_c\otimes w_c\|_\alpha\leq \\
&\leq C(\|v_\ell +w\|_0\|w_c\|_\alpha+\|v_\ell +w\|_\alpha\|w_c\|_0+\|w_c\|_0\|w_c\|_\alpha)\, .\\
&\leq C \|w_c\|_\alpha \left(\|v\|_0 + \|w_o\|_\alpha + \|w_c\|_\alpha\right)\, .
\end{split}
\end{equation*}
From Proposition \ref{p:velocity_est} we then conclude
\[
\|(v_\ell +w) \otimes w_c+w_c\otimes (v_\ell +w)-w_c\otimes w_c\|_\alpha
\leq C_{v,s} \sqrt{\delta} D \mu \lambda^{2\alpha-1}\,  .
\]
By the Schauder estimate \eqref{e:Schauder_Rdiv}, we get
\begin{equation}\label{e:five}
\|\mathring{R}^5_1\|_0 \leq C_{v,s} \sqrt{\delta} D \mu \lambda^{2\alpha-1} .
\end{equation}
As for $\|\mathring{R}^5_1\|_1$ the same argument yields 
\[
\|\mathring{R}^5_1\|_1 \leq C_{v,s} \sqrt{\delta} D \mu \lambda^{2\alpha}\, .
\] 
Finally 
\begin{align}
&\|\partial_t ((v_\ell + w) \otimes w_c+w_c\otimes (v_\ell +w)-w_c\otimes w_c)\|_\alpha\nonumber\\
\leq &\|w_c\|_\alpha \left( \|\partial_t v_\ell \|_\alpha + \|\partial_t w_o\|_\alpha + \|\partial_t w_c\|_\alpha\right) 
 + \|\partial_t w_c\|_\alpha \left( \|v_\ell\|_\alpha + \|w_o\|_\alpha\right)\nonumber\\
\leq &C_{e,s} \sqrt{\delta} D \mu \lambda^{\alpha-1} \left( C_h D \ell^{-\alpha} +  C_v \sqrt{\delta} \lambda^{1+\alpha}
+ C_{v,s}\sqrt{\delta} D \mu \lambda^\alpha\right)\nonumber\\
& + C_{v,s} D \mu \lambda^\alpha \left(C_h D \ell^{1-\alpha} + C_e \sqrt{\delta} \lambda^\alpha\right)  
\leq C_{v,s} \delta D  \mu \lambda^{2\alpha}\, .
\end{align}
Hence we conclude
\begin{equation}\label{e:five_D}
\|\mathring{R}^5_1\|_{C^1} \leq C_{v,s} \sqrt{\delta} D \mu \lambda^{2\alpha} \, .
\end{equation}

\medskip

{\bf Step 6.} In this step we argue as in \cite[Lemma 7.5]{DS3}. 
Since $B_k\cdot k=0$, we can write
\begin{equation*}
\begin{split}
\div(v_\ell \otimes w_o)&=(w_o\cdot\nabla) v_\ell +(\div w_o)v_\ell \\
&=\sum_{|k|=\lambda_0}\left[a_k(B_k\cdot \nabla)v_{\ell}+v_{\ell}
(B_k\cdot\nabla) a_k\right]e^{i\lambda k\cdot x}\, .
\end{split}
\end{equation*}
Choose $m$ as in \eqref{e:choice_of_m}, apply Propositions \ref{p:schauder}
and \ref{p:W} and use \eqref{e:reason_for_m} to get
\begin{align}
\|\mathring{R}^6_1\|_0 &\leq C_s \sum_{|k|=\lambda_0}
\lambda^{\alpha-1} \left(\|a_k\|_0 [v_\ell]_1 + \|v_\ell\|_0 [a_k]_1\right)
\nonumber\\
&\quad + C_s \sum_{|k|=\lambda_0}\lambda^{-m+\alpha} \left(\|a_k\|_0 [v_\ell]_{m+1} + \|v_\ell\|_0 [a_k]_{m+1}\right)\nonumber\\
&\quad+ C_s \sum_{|k|=\lambda_0} \lambda^{-m} \left(\|a_k\|_0 [v_\ell]_{m+1+\alpha}
+ \|v_\ell\|_0 [a_k]_{m+1+\alpha}\right)\nonumber\\
&\leq C_{v,s} \lambda^{\alpha-1} \sqrt{\delta} (D +D\mu) + C_{v,s} \lambda^{ -m+\alpha} \sqrt{\delta}
\left(D\ell^{-m} + D^{m+1} \mu^{m+1}\right)\nonumber\\
&\quad + C_{v,s} \lambda^{-m} \sqrt{\delta} \left( D \ell^{-m-\alpha}
+ D^{m+1+\alpha} \mu^{m+1+\alpha}\right)\nonumber\\
&\leq C_{v,s}\sqrt{\delta} D \mu \lambda^{\alpha-1}\, .\label{e:six}
\end{align}
As in the Steps 3 and 4:
\begin{align}
\|\mathring{R}^6_1\|_1 &\leq  C_e \lambda \|\mathring{R}^6_1\|_0
+ C_s \sum_{|k|=\lambda_0}
\lambda^{\alpha-1} \left(\|a_k\|_0 [v_\ell]_2 + \|v_\ell\|_0 [a_k]_2\right)
\nonumber\\
&\quad + C_s \sum_{|k|=\lambda_0}\lambda^{\alpha-m} \left(\|a_k\|_0 [v_\ell]_{m+2} + \|v_\ell\|_0 [a_k]_{m+2}\right)\nonumber\\
&\quad+ C_s \sum_{|k|=\lambda_0} \lambda^{-m} \left(\|a_k\|_0 [v_\ell]_{m+2+\alpha}
+ \|v_\ell\|_0 [a_k]_{m+2+\alpha}\right)\nonumber\\
&\leq C_{v,s} \sqrt{\delta} D \mu \lambda^\alpha\, .\label{e:dwarf_six}
\end{align}
As for the time derivative, we can estimate 
\[
\|\partial_t \mathring{R}^6_1\|_0 \leq {\rm (I)} + {\rm (II)} + {\rm (III)}\, ,
\]
where
\begin{align}
{\rm (I)} &= C_s \sum_{|k|=\lambda_0} \lambda^\alpha
\left(\|\partial_\tau a_k\|_0 [v_\ell]_1 + \|v_\ell\|_0 [\partial_\tau a_k]_1\right)
\nonumber\\
&\quad+ C_s \sum_{|k|=\lambda_0}
\lambda^{\alpha-1} \left(\|\partial_s a_k\|_0 [v_\ell]_1 + \|v_\ell\|_0 [\partial_s a_k]_1\right)
\nonumber\\
&\quad + C_s \sum_{|k|=\lambda_0}
\lambda^{\alpha-1} \left(\|a_k\|_0 [\partial_t v_\ell]_1 + \|\partial_t v_\ell\|_0 [a_k]_1\right)\, ,
\end{align}
\begin{align}
{\rm (II)} &= C_s \sum_{|k|=\lambda_0}\lambda^{\alpha+ 1-m} \left(\|\partial_\tau a_k\|_0 [v_\ell]_{m+1} + \|v_\ell\|_0 [\partial_\tau a_k]_{m+1}\right)\nonumber\\
&\quad + C_s \sum_{|k|=\lambda_0}\lambda^{\alpha-m} \left(\|\partial_s a_k\|_0 [v_\ell]_{m+1} + \|v_\ell\|_0 [\partial_s a_k]_{m+1}\right)\nonumber\\
&\quad + C_s \sum_{|k|=\lambda_0}\lambda^{\alpha -m} \left(\|a_k\|_0 [\partial_t v_\ell]_{m+1} + \|\partial_t v_\ell\|_0 [a_k]_{m+1}\right)
\end{align}
and
\begin{align}
{\rm (III)} &= C_s \sum_{|k|=\lambda_0} \lambda^{1-m} \left(\|\partial_\tau a_k\|_0 [v_\ell]_{m+1+\alpha}
+ \|v_\ell\|_0 [\partial_\tau a_k]_{m+1+\alpha}\right)\nonumber\\
&\quad+ C_s \sum_{|k|=\lambda_0} \lambda^{-m} \left(\|\partial_s a_k\|_0 [v_\ell]_{m+1+\alpha}
+ \|v_\ell\|_0 [\partial_s a_k]_{m+1+\alpha}\right)\nonumber\\
&\quad+ C_s \sum_{|k|=\lambda_0} \lambda^{-m} \left(\|a_k\|_0 [\partial_t v_\ell]_{m+1+\alpha}
+ \|\partial_t v_\ell\|_0 [a_k]_{m+1+\alpha}\right)\, .
\end{align}
Again using Proposition \ref{p:W} and the conditions \eqref{e:range} we can see that
\begin{equation}\label{e:dwarf_six_2}
\|\partial_t \mathring{R}^6_1\|_0\leq C_{v,s} \sqrt{\delta} D \mu \lambda^\alpha\, .
\end{equation}
Thus,
\begin{equation}\label{e:six_D}
\|\mathring{R}^6_1\|_{C^1} \leq \|\mathring{R}^6_1\|_1 +
\|\partial_t \mathring{R}^6_1\|_0 \leq C_{v,s} \sqrt{\delta} D\mu \lambda^\alpha\, .
\end{equation}

\medskip

{\bf Step 7.} Finally, to bound the last term we argue as in \cite[Lemma 7.1]{DS3}.
We write
\begin{equation*}
\mathring{R}^7_1=\RR(\partial_tw_o+v_\ell \cdot \nabla w_o)=\mathring{R}^8_1+\mathring{R}^9_1+\mathring{R}^{10}_1,
\end{equation*}
where
\begin{align*}
\mathring{R}^8_1&:=\lambda\RR\left(\sum_{|k|=\lambda_0}(\partial_\tau a_k+i(k\cdot v_\ell )a_k)(x,t,\lambda t)B_ke^{i\lambda k\cdot x}\right)\\
\mathring{R}^9_1&:=\RR\left(\sum_{|k|=\lambda_0}(\partial_sa_k)(x,t,\lambda t)B_ke^{i\lambda k\cdot x}\right)\\
\mathring{R}^{10}_1&:=\RR \left(\sum_{|k|=\lambda_0}(v_\ell \cdot \nabla_y a_k) (x,t,\lambda t) B_ke^{i\lambda k\cdot x}\right)\, .
\end{align*}
The arguments of Step 6 have already shown
\begin{align}
\|\mathring{R}^{10}_1\|_0 &\leq C_{v,s} \sqrt{\delta} D \mu \lambda^{\alpha-1}\label{e:ten}\\
\|\mathring{R}^{10}_1\|_{C^1} &\leq C_{v,s} \sqrt{\delta} D \mu \lambda^{\alpha}\label{e:ten_D}\, .
\end{align}
As for $\mathring{R}^9_1$, we apply Proposition \ref{p:schauder} with $m$ as in \eqref{e:choice_of_m}
to get
\begin{align}
\|\mathring{R}^9_1\|_0 &\leq C_s \sum_{|k|=\lambda_0} \left(\lambda^{\alpha-1} \|\partial_s a_k\|_0
+ \lambda^{-m+\alpha} [\partial_s a_k]_m + \lambda^{-m} [\partial_s a_k]_{m+\alpha}\right)\nonumber\\
&\leq C_{e,s} \sqrt{\delta} D \mu \lambda^{\alpha-1}\label{e:nine}\, .
\end{align}
Analogously
\begin{align}
\|\mathring{R}^9_1\|_1 &\leq C_e \lambda \|\mathring{R}^9_1\|_0\nonumber\\
&\quad + C_s \sum_{|k|=\lambda_0} \left(\lambda^{\alpha-1} [\partial_s a_k]_1
+ \lambda^{-m+\alpha} [\partial_s a_k]_{m+1} + \lambda^{-m} [\partial_s a_k]_{m+1+ \alpha}\right)\nonumber\\
&\leq C_{e,s}\sqrt{\delta} D \mu \lambda^{\alpha}\label{e:dwarf_nine}\, 
\end{align}
and
\begin{align}
\|\partial_t \mathring{R}^9_1\|_0 &\leq C_s \sum_{|k|=\lambda_0} \left(\lambda^{\alpha-1} \|\partial_{ss} a_k\|_0
+ \lambda^{-m+\alpha} [\partial_{ss} a_k]_m + \lambda^{-m} [\partial_{ss} a_k]_{m+\alpha}\right)\nonumber\\
&\quad + C_s \sum_{|k|=\lambda_0} \left(\lambda^{\alpha} \|\partial_{s\tau} a_k\|_0
+ \lambda^{1-m+\alpha} [\partial_{s\tau} a_k]_m + \lambda^{1-m} [\partial_{s\tau} a_k]_{m+\alpha}\right)\nonumber\\
&\leq C_{v,s} \sqrt{\delta} D \mu \lambda^{\alpha}\label{e:dwarf_nine_2}\, ,
\end{align}
which in turn imply
\begin{equation}\label{e:nine_D}
\|\mathring{R}^9_1\|_{C^1}\leq C_{v,s} \sqrt{\delta} D \mu \lambda^\alpha\, .
\end{equation}
For the term $\mathring{R}^8_1$ define the functions
\[
b_k (y,s, \tau) := (\partial_\tau a_k + i (k\cdot v_\ell) a_k) (y,s,\tau)\, .
 \]
Applying Proposition \ref{p:schauder} with $m$ as in \eqref{e:choice_of_m} then yields
\begin{align}
\|\mathring{R}^8_1\|_0
&\leq C_s \sum_{|k|=\lambda_0} \left(\lambda^{\alpha} \|b_k\|_0 + \lambda^{\alpha+1-m} [b_k]_m+
\lambda^{1-m} [b_k]_{m+\alpha}\right)\nonumber\\
&\leq C_{e,s} \sqrt{\delta} \mu^{-1} \lambda^\alpha  + C_{e,s} \sqrt{\delta} \left(\mu^{m-1} D^m + D \ell^{1-m}\right)
\lambda^{\alpha+1-m}\nonumber\\ 
&\quad + C_{e,s} \sqrt{\delta} \left( \mu^{m-1+\alpha} D^{m+\alpha} + D \ell^{1-m-\alpha}\right)
\lambda^{1-m}\nonumber\\
&\leq C_{e,s} \sqrt{\delta} \mu^{-1} \lambda^\alpha\, .
\end{align}
Similarly,
\begin{align}
\|\mathring{R}^8_1\|_1
&\leq C_e \lambda \|\mathring{R}^8_1\|_0\nonumber\\
&+C_s \sum_{|k|=\lambda_0} \left(\lambda^{\alpha} [b_k]_1 + \lambda^{\alpha+1-m} [b_k]_{m+1}+
\lambda^{1-m} [b_k]_{m+1+\alpha}\right)\nonumber\\
&\leq C_{e,s} \sqrt{\delta} \mu^{-1} \lambda^{1+\alpha}\, .\label{e:dwarf_8}
\end{align}
Finally, differentiating $\mathring{R}^8_1$ in time and using the same arguments:
\begin{align}
\|\partial_t \mathring{R}^8_1\|_0
&\leq C_s \lambda \sum_{|k|=\lambda_0} \left(\lambda^{\alpha} \|\partial_\tau b_k\|_0 + \lambda^{\alpha+1-m} [\partial_\tau b_k]_{m}+
\lambda^{1-m} [\partial_\tau b_k]_{m+\alpha}\right)\nonumber\\
&\quad + C_s \sum_{|k|=\lambda_0} \left(\lambda^{\alpha} \|\partial_s b_k\|_0 + \lambda^{\alpha+1-m} [\partial_s b_k]_{m}+
\lambda^{1-m} [\partial_s b_k]_{m+\alpha}\right)\nonumber\\
&\leq C_{v,s} \sqrt{\delta} \mu^{-1} \lambda^{1+\alpha}\, .\label{e:dwarf_8_2}
\end{align}
Therefore
\begin{equation}\label{e:eight_D}
\|\mathring{R}^8_1\|_{C^1} \leq C_{v,s} \sqrt{\delta} \mu^{-1} \lambda^{1+\alpha}\, .
\end{equation}
Summarizing
\begin{align}
\|\mathring{R}^7_1\|_0\leq C_{v,s} \sqrt{\delta} \left(D \mu \lambda^{\alpha-1} + \mu^{-1} \lambda^\alpha\right)\label{e:seven}\\
\|\mathring{R}^7_1\|_{C^1}\leq C_{v,s} \sqrt{\delta} \left(D \mu \lambda^{\alpha} + \mu^{-1} \lambda^{\alpha+1}\right)\, .\label{e:seven_D}
\end{align}

\medskip

{\bf Conclusion.} From \eqref{e:one}, \eqref{e:two},
\eqref{e:three}, \eqref{e:four}, \eqref{e:five}, \eqref{e:six} and \eqref{e:seven},
we conclude
\begin{align}
\|\mathring{R}_1\|_0 &\leq C_{v,s} \Big( D \ell + \sqrt{\delta} D\ell + \delta D \mu \lambda^{\alpha-1}
+ \sqrt{\delta} D \mu \lambda^{\alpha-1}\nonumber\\
&\qquad\quad  + \sqrt{\delta} D \mu \lambda^{2\alpha-1} + \sqrt{\delta}\mu^{-1}
\lambda^\alpha\Big)\nonumber\\
&\leq C_{v,s} \left(D\ell + \sqrt{\delta} D \mu \lambda^{2\alpha-1} + \sqrt{\delta}  \mu^{-1}\lambda^\alpha\right)\, .
\end{align}
From \eqref{e:one_D}, \eqref{e:two_D},
\eqref{e:three_D}, \eqref{e:four_D}, \eqref{e:five_D}, \eqref{e:six_D} and \eqref{e:seven_D},
we conclude
\begin{align}
\|\mathring{R}_1\|_{C_1} &\leq C_{v,s} \Big( D + \sqrt{\delta} \lambda D\ell + \delta D \mu \lambda^{\alpha}
+ \sqrt{\delta} D \mu \lambda^{\alpha}\nonumber\\
&\qquad\quad  + \sqrt{\delta} D \mu \lambda^{2\alpha} + \sqrt{\delta}\mu^{-1}
\lambda^{1+\alpha}\Big)\nonumber\\
&\leq C_{v,s} \left(\sqrt{\delta} D \ell \lambda + \sqrt{\delta} D \mu \lambda^{2\alpha} + \sqrt{\delta} D \mu^{-1}\lambda^{\alpha+1}\right)\,.
\end{align}
In the last inequality we have used \eqref{e:range} once more: $\sqrt{\delta}\mu D\geq D\delta^{-1/2}\geq D$.
\end{proof}

\section{Proof of Proposition \ref{p:iterate}}\label{s:iterate_proof}

{\bf Step 1.} We now specify the choice of the parameters, in the order in which they are chosen.
Recall that $\eps$ is a fixed positive number, given by the proposition. The exponent 
$\omega$ has already been chosen according to
\begin{equation}\label{e:choice_of_omega}
1+\eps = \frac{1+\omega}{1-\omega}\, .
\end{equation}
Next we choose a suitable exponent $\alpha$ for which we can apply the Propositions \ref{p:energy} and 
\ref{p:Reynolds}. To be precise we set
\begin{equation}\label{e:choice_of_alpha}
\alpha = \frac{\omega}{2(1+\omega)}\, .
\end{equation}
The reason for these choices will become clear in the following. For the moment we just observe that
both $\alpha$ and $\omega$ depend only on $\varepsilon$ and that
$\alpha\in (0, \frac{\omega}{1+\omega})$, i.e. both Propositions \ref{p:energy}
and \ref{p:Reynolds} are applicable.

We next choose:
\begin{equation}\label{e:choice_of_ell}
\ell = \frac{1}{L_v}\frac{\bar{\delta}}{D}\, 
\end{equation}
with $L_v$ being a sufficiently large constant, which depends only on $\|v\|_0$
and $e$. 

Next, we impose
\begin{equation}\label{e:choice_of_mu}
\mu^2 D = \lambda\,
\end{equation} 
and
\begin{equation}\label{e:choice_of_lambda}
\lambda = \Lambda_v \left( \frac{D\delta}{\bar{\delta}^2}\right)^{\frac{1}{1-4\alpha}}
= \Lambda_v \left(\frac{D\delta}{\bar{\delta}^2}\right)^{\frac{1+\omega}{1-\omega}}
= \Lambda_v \left(\frac{D\delta}{\bar{\delta}^2}\right)^{1+\eps}\, ,
\end{equation}
where $\Lambda_v$ is a sufficiently large constant, which depends only on $\|v\|_0$.
Concerning the constants $L_v$ and $\Lambda_v$ we will see that they will be chosen in this order in Step 3 below.
Observe also that $\mu$, $\lambda$ and $\frac{\lambda}{\mu}$ must be integers. However, this can be
reached by imposing the less stringent constraints
\[
\frac{\lambda}{2} \leq \mu^2 D \leq \lambda
\]
and
\[
\Lambda_v \left(\frac{D\delta}{\bar{\delta}^2}\right)^{1+\eps}\leq
\lambda \leq 2 \Lambda_v \left(\frac{D\delta}{\bar{\delta}^2}\right)^{1+\eps}\, ,
\]
provided $\Lambda_v$ is larger than some universal constant. This would require just minor
adjustments in the rest of the argument.

\medskip

{\bf Step 2. Compatibility conditions.} We next check that all the conditions in \eqref{e:range} are satisfied by our choice
of the parameters.

First of all, since $\bar{\delta}\leq \delta$, the inequality $\ell^{-1}\geq \frac{D}{\eta\delta}$ is for sure
achieved if we impose
\begin{equation}\label{e:Lv}
L_v\geq \eta^{-1}\, .
\end{equation}

Next, \eqref{e:choice_of_lambda} and $\Lambda_v\geq 1$ implies
\[
\mu = \sqrt{\frac{\lambda}{D}} \geq \frac{\sqrt{\delta}}{\bar{\delta}} \geq \delta^{-1}
\]
because by assumption $\bar{\delta} \leq \delta^{\frac{3}{2}}$.

Also,
\[
\frac{\lambda}{(\mu D)^{1+\omega}} = \frac{\lambda^{\frac{1-\omega}{2}}}{D^{\frac{1+\omega}{2}}}
= \Lambda_v^{\frac{1-\omega}{2}} \left(\frac{\delta}{\bar{\delta}^2}\right)^{\frac{1+\omega}{2}}\, .
\]
Since $\omega<1$, $\Lambda_v\geq 1$ and $\bar{\delta}\leq \delta$, we conclude
$\lambda \geq (\mu D)^{1+\omega}$. Finally
\[
\lambda \ell^{1+\omega} = \Lambda_v \left(\frac{D\delta}{\bar{\delta}^2}\right)^{\frac{1+\omega}{1-\omega}}
\left(L_v^{-1} \bar{\delta} D^{-1}\right)^{1+\omega}
= \frac{\Lambda_v}{L_v^{1+\omega}} \left(D^\omega \frac{\delta}{\bar{\delta}^{1+\omega}}\right)^{\frac{1+\omega}{1-\omega}}\, .
\]
Thus, by requiring 
\begin{equation}\label{e:Lambdav}
\Lambda_v \geq L_v^{1+\omega}
\end{equation}
we satisfy $\lambda \geq \ell^{-(1+\omega)}$.
Hence, all the requirements in \eqref{e:range} are satisfied provided that the constants $L_v$ and $\Lambda_v$ are chosen to satisfy \eqref{e:Lv} and \eqref{e:Lambdav}.

\medskip

{\bf Step 3. $C^0$ estimates.} Having verified that $\alpha\in (0, \frac{\omega}{1+\omega})$ and
that \eqref{e:range} holds, we can apply the Propositions \ref{p:velocity_est},
\ref{p:energy} and \ref{p:Reynolds}. Proposition \ref{p:Reynolds} implies
\begin{align}
\|\mathring{R}_1\|_0 &\leq C_v \left(D\ell + \sqrt{\delta} D^{\frac{1}{2}} \lambda^{2\alpha-\frac{1}{2}} + \sqrt{\delta} D^{\frac{1}{2}} \lambda^{\alpha-\frac{1}{2}}\right)\nonumber\\
&\leq \frac{C_v}{L_v} \bar{\delta} + \frac{C_v}{\Lambda_v^{\frac{1+\eps}{2}}}
\bar{\delta}\, 
\end{align}
(since now the exponent $\alpha$ has been fixed, we can forget about the $\alpha$-dependence
of the constants in the estimates of Proposition \ref{p:energy} and \ref{p:Reynolds}).
Choosing first $L_v$ and, then, $\Lambda_v$ sufficiently large, we can achieve the desired inequalities \eqref{e:Lv}-\eqref{e:Lambdav} together with 
\[
\|\mathring{R}_1\|_0 \leq \eta \bar{\delta}\, .
\]
Next, using Proposition \ref{p:energy}, it is also easy to check that, by this choice,
\eqref{e:energy_est} is satisfied as well. Furthermore, recall that, by Proposition \ref{p:velocity_est}, 
\[
\|v_1-v\|_0 = \|w\|_0 \leq C_e\sqrt{\delta}\, .
\]
If we impose $M$ to be larger than this particular constant $C_e$ (which depends only on $e$), we then achieve \eqref{e:C^0_est}.

Finally, as already observed in \eqref{e:pressure_M},
\[
\|p_1-p\|_0 = \frac{M^2}{4} \delta + \|v-v_\ell\|_0 \|w\|_0\, .
\]
Since $\|v-v_\ell\|_0\leq C D \ell \leq C  \bar{\delta}$
and $\|w\|_0\leq  C_e \sqrt{\delta}$, we easily conclude the inequality
\eqref{e:C0_pressure}. This completes the proof of all the conclusions of Proposition \ref{p:iterate}
except for the estimate of $\max\{\|v_1\|_{C^1}, \|\mathring{R}_1\|_{C^1}\}$.

\medskip

{\bf Step 4. $C^1$ estimates.} By Proposition \ref{p:Reynolds} and
the choices specified above we also have
\[
\|\mathring{R}_1\|_{C^1} \leq \bar{\delta} \lambda
\]
whereas Proposition \ref{p:velocity_est} shows
\begin{align*}
\|v_1\|_{C^1} &\leq D + \|w\|_{C^1} \leq D + C_e \sqrt{\delta} \lambda  .
\end{align*}
Thus, we conclude
\begin{align*}
\max \left\{\|v_1\|_{C^1}, \|\mathring{R}_1\|_{C^1}\right\}
&\leq D + C_e \sqrt{\delta} \lambda \leq D + C_e \sqrt{\delta} \Lambda_v  \left(
\frac{D\delta}{\bar{\delta}^2}\right)^{1+\eps}\nonumber\\
&\leq D + C_e \Lambda_v  \delta^{\frac{3}{2}} \left(\frac{D}{\bar{\delta}^2}\right)^{1+\eps}\, .
\end{align*}
Since $\delta^{\frac{3}{2}}\geq \bar{\delta}^2$, we obtain
\[
\max \left\{\|v_1\|_{C^1}, \|\mathring{R}_1\|_{C^1}\right\}
\leq 2 C_e \Lambda_v \delta^{\frac{3}{2}}\left(\frac{D}{\bar{\delta}^2}\right)^{1+\eps}\, .
\]
Setting $A = 2 C_e \Lambda_v$, we conclude estimate \eqref{e:C1_est}.

\section{Proof of Remark \ref{r:time_and_pressure}}\label{s:final}

{\bf Step 1. Estimate on the $C^1$ norm.} We claim that the proof of Proposition \ref{p:iterate} yields also the estimate
\begin{equation}\label{e:C1_pressure}
\|p_1\|_{C^1}\leq \|p\|_{C^1} + A \delta^{2+\eps} \left(\frac{D}{\bar{\delta}^2}\right)^{1+\eps}\, ,
\end{equation}
where, as in Proposition \ref{p:iterate}, $A$ is a constant which depends only on $e$, $\eps>0$ and $\|v\|_0$.
Indeed, recall the formula for the pressure:
\[
p_1 = p - \frac{|w_o|^2}{2} - \langle v-v_\ell, w\rangle\, .
\]
Therefore we estimate, using Proposition \ref{p:velocity_est}
\begin{align*}
\|p_1\|_{C^1} - \|p\|_{C^1} &\leq \|w_o\|_0 \|w_o\|_{C^1} + \|w\|_0 \|v-v_\ell\|_{C^1} + \|w\|_{C^1} \|v-v_\ell\|_0\\
&\leq C_e \delta \lambda + C_e D \sqrt{\delta}+ C_e D \ell \sqrt{\delta}\lambda\,.
\end{align*}
As before, \eqref{e:range} implies $\lambda\geq \mu D\geq D\delta^{-1}$ and $D\ell\leq \delta$. 
Therefore, we conclude
\begin{align*}
\|p_1\|_{C^1}&\leq \|p\|_{C^1} + C_e \delta \lambda \leq \|p\|_{C^1} + C_e \Lambda_v \delta \left(\frac{D\delta}{\bar{\delta}^2}\right)^{1+\eps}\\
&\leq \|p\|_{C^1} + A \delta^{2+\eps} \left(\frac{D}{\bar{\delta}^2}\right)^{1+\eps}\, .
\end{align*}

\medskip

{\bf Step 2. Iteration.} We now proceed as in the proof of Theorem \ref{t:main}. We construct 
the sequence $(p_n, v_n, \mathring{R}_n)$ of solutions to the Euler-Reynolds system, starting
from $(p_0, v_0, \mathring{R}_0)=(0,0,0)$ and applying Proposition \ref{p:iterate} with
$\delta_n = a^{-b^n}$. As in the proof of Theorem \ref{t:main}, we set 
\[
b=\frac{3}{2},\quad c= \frac{3(1+2\eps)}{1-2\eps}+\eps
\]
and choose $a$ sufficiently large so to guarantee the inequality
$$
D_n = \max\{\|v_n\|_{C^1}, \|\mathring{R}_n\|_{C^1}\} \leq a^{cb^n}\,.
$$
We then use 
\eqref{e:C1_pressure} to conclude
\[
\|p_{n+1}\|_{C^1} \leq \|p_n\|_{C^1} + A a^{(1+2\eps) (c+1) b^n}\, .
\]
Since $A$ depends only on $\|v_n\|_0$ which turns out to be uniformly
bounded, we can assume that $A$ does not depend on $n$. Therefore, 
if we choose $a$ sufficiently large, we can then write
\[
\|p_{n+1}\|_{C^1}\leq \|p_n\|_{C^1} + a^{(1+3\eps)(c+1)b^n}
\]
Since $p_0=0$, we inductively get the estimate 
\[
\|p_{n+1} \|_{C^1}\leq (n+1) a^{(1+3\eps)(c+1)b^n}\leq a^{[(1+4\eps)(c+1)]b^n}\, 
\]
(again the last inequality is achieved choosing $a$ sufficiently large).
Summarizing, if we set $\vartheta=(1+4\eps)(c+1)$, we have
\begin{align*}
\|p_{n+1}-p_n\|_0 \leq C_e \delta_n \leq C_e a^{-b^n}\\
\|p_{n+1}-p_n\|_{C^1}\leq a^{\vartheta b^n}
\end{align*}
Interpolating we get $\|p_{n+1}-p_n\|_{C^\varrho} \leq C_e a^{(\varrho (1+\vartheta) -1) b^n}$
for every $\varrho\in (0,1)$.
Thus the limiting pressure $p$ belongs to $C^\varrho$ for every
\[
\varrho <\frac{1}{1+\vartheta}= \frac{1}{1 + (1+4\eps)(c+1)}\, .
\]
As $\eps\downarrow 0$, we have $c\downarrow 3$ and hence
\[
\frac{1}{1+\vartheta} \uparrow \frac{1}{5}\, .
\]
Therefore, for every $\theta<\frac{1}{10}$, if the $\eps$ in Proposition \ref{p:iterate} is chosen
sufficiently small, we construct a pair $(p,v)$ which satisfies the conclusion of Theorem \ref{t:main}
and belongs to $C^\theta (\T^3\times \mathbb S^1, \R^3)\times C^{2\theta} (\T^3\times \mathbb S^1)$.


\begin{thebibliography}{10}

\bibitem{BorisovRigidity1}
{\sc Borisov, J.~F.}
\newblock On the connection bewteen the spatial form of smooth surfaces and
  their intrinsic geometry.
\newblock {\em Vestnik Leningrad. Univ. 14}, 13 (1959), 20--26.

\bibitem{Borisov65}
{\sc Borisov, J.~F.}
\newblock {$C^{1,\alpha}$-isometric immersions of {Riemannian} spaces}.
\newblock {\em Doklady 163\/} (1965), 869--871.

\bibitem{CCFS2007}
{\sc Cheskidov, A., Constantin, P., Friedlander, S., and Shvydkoy, R.}
\newblock Energy conservation and {O}nsager's conjecture for the {E}uler
  equations.
\newblock {\em Nonlinearity 21}, 6 (2008), 1233--1252.

\bibitem{CFP2007}
{\sc Cheskidov, A., Friedlander, S., and Pavlovi{\'c}, N.}
\newblock Inviscid dyadic model of turbulence: the fixed point and {O}nsager's
  conjecture.
\newblock {\em J. Math. Phys. 48}, 6 (2007), 065503, 16.

\bibitem{CFP2010}
{\sc Cheskidov, A., Friedlander, S., and Pavlovi{\'c}, N.}
\newblock An inviscid dyadic model of turbulence: the global attractor.
\newblock {\em Discrete Contin. Dyn. Syst. 26}, 3 (2010), 781--794.

\bibitem{Constantin1997}
{\sc Constantin, P.}
\newblock {The Littlewood-Paley spectrum in two-dimensional turbulence.}
\newblock {\em Theor. Comput. Fluid Dyn. 9}, 3-4 (1997), 183--189.

\bibitem{ConstantinETiti}
{\sc Constantin, P., E, W., and Titi, E.~S.}
\newblock Onsager's conjecture on the energy conservation for solutions of
  {E}uler's equation.
\newblock {\em Comm. Math. Phys. 165}, 1 (1994), 207--209.

\bibitem{CDSz}
{\sc Conti, S., De~Lellis, C., and Sz{\'e}kelyhidi, Jr., L.}
\newblock $h$-principle and rigidity for {$C^{1,\alpha}$} isometric embeddings.
\newblock {\em To appear in the Proceedings of the Abel Symposium 2010\/}
  (2011).

\bibitem{DS1}
{\sc De~Lellis, C., and Sz{\'e}kelyhidi, Jr., L.}
\newblock The {E}uler equations as a differential inclusion.
\newblock {\em Ann. of Math. (2) 170}, 3 (2009), 1417--1436.

\bibitem{DS2}
{\sc De~Lellis, C., and Sz{\'e}kelyhidi, Jr., L.}
\newblock On admissibility criteria for weak solutions of the {E}uler
  equations.
\newblock {\em Arch. Ration. Mech. Anal. 195}, 1 (2010), 225--260.

\bibitem{DSSurvey}
{\sc De~Lellis, C., and Sz{\'e}kelyhidi, Jr., L.}
\newblock The $h$-principle and the equations of fluid dynamics.
\newblock {\em Preprint. To appear in the {\em Bull. of the Amer. Math.
  Soc.}\/} (2011).

\bibitem{DS3}
{\sc De~Lellis, C., and Sz{\'e}kelyhidi, Jr., L.}
\newblock Continuous dissipative {E}uler flows.
\newblock {\em Preprint\/} (2012).

\bibitem{RobertDuchon}
{\sc Duchon, J., and Robert, R.}
\newblock Inertial energy dissipation for weak solutions of incompressible
  {E}uler and {N}avier-{S}tokes equations.
\newblock {\em Nonlinearity 13}, 1 (2000), 249--255.

\bibitem{Eyink}
{\sc Eyink, G.~L.}
\newblock Energy dissipation without viscosity in ideal hydrodynamics. {I}.
  {F}ourier analysis and local energy transfer.
\newblock {\em Phys. D 78}, 3-4 (1994), 222--240.

\bibitem{EyinkSreenivasan}
{\sc Eyink, G.~L., and Sreenivasan, K.~R.}
\newblock Onsager and the theory of hydrodynamic turbulence.
\newblock {\em Rev. Modern Phys. 78}, 1 (2006), 87--135.

\bibitem{FrischBook}
{\sc Frisch, U.}
\newblock {\em Turbulence}.
\newblock Cambridge University Press, Cambridge, 1995.
\newblock The legacy of A. N. Kolmogorov.

\bibitem{Gromov}
{\sc Gromov, M.}
\newblock {\em Partial differential relations}, vol.~9 of {\em Ergebnisse der
  Mathematik und ihrer Grenzgebiete (3)}.
\newblock Springer-Verlag, Berlin, 1986.

\bibitem{KatzPavlovic}
{\sc Katz, N.~H., and Pavlovi{\'c}, N.}
\newblock Finite time blow-up for a dyadic model of the {E}uler equations.
\newblock {\em Trans. Amer. Math. Soc. 357}, 2 (2005), 695--708 (electronic).

\bibitem{Kolmogorov}
{\sc Kolmogorov, A.~N.}
\newblock The local structure of turbulence in incompressible viscous fluid for
  very large {R}eynolds numbers.
\newblock {\em Proc. Roy. Soc. London Ser. A 434}, 1890 (1991), 9--13.
\newblock Translated from the Russian by V. Levin, Turbulence and stochastic
  processes: Kolmogorov's ideas 50 years on.

\bibitem{LichtensteinBook}
{\sc Lichtenstein, L.}
\newblock {\em Grundlagen der {H}ydromechanik}.
\newblock Die Grundlehren der mathematischen Wissenschaften in
  Einzeldarstellungen, Band 30. Springer-Verlag, Berlin, 1968.

\bibitem{Nash}
{\sc Nash, J.}
\newblock {$C^1$} isometric imbeddings.
\newblock {\em Ann. of Math. (2) 60\/} (1954), 383--396.

\bibitem{Onsager}
{\sc Onsager, L.}
\newblock Statistical hydrodynamics.
\newblock {\em Nuovo Cimento (9) 6}, Supplemento, 2 (Convegno Internazionale di
  Meccanica Statistica) (1949), 279--287.

\bibitem{OseenBook}
{\sc Oseen, C.~W.}
\newblock {\em Neuere {M}ethoden und {E}rgebnisse in der {H}ydrodynamik}.
\newblock Mathematik und ihre Anwendungen in Monographien und Lehrb{\"u}chern.
  Leipzig : Akademische Verlagsgesellschaft, Leipzig, 1927.

\bibitem{Robert}
{\sc Robert, R.}
\newblock Statistical hydrodynamics ({O}nsager revisited).
\newblock In {\em Handbook of mathematical fluid dynamics, {V}ol. {II}}.
  North-Holland, Amsterdam, 2003, pp.~1--54.

\bibitem{Scheffer93}
{\sc Scheffer, V.}
\newblock An inviscid flow with compact support in space-time.
\newblock {\em J. Geom. Anal. 3}, 4 (1993), 343--401.

\bibitem{Shnirelman1}
{\sc Shnirelman, A.}
\newblock On the nonuniqueness of weak solution of the {E}uler equation.
\newblock {\em Comm. Pure Appl. Math. 50}, 12 (1997), 1261--1286.

\bibitem{Shnirelmandecrease}
{\sc Shnirelman, A.}
\newblock Weak solutions with decreasing energy of incompressible {E}uler
  equations.
\newblock {\em Comm. Math. Phys. 210}, 3 (2000), 541--603.

\end{thebibliography}

\end{document}